\documentclass[10pt]{article}
\usepackage{authblk}
\usepackage{amssymb}
\usepackage[all]{xy}
\usepackage{mathrsfs}
\usepackage{amsmath,amssymb,amsthm}
\usepackage{extarrows}
\usepackage{mathrsfs}
\newtheorem{theorem}{Theorem}[section]
\newtheorem{lemma}[theorem]{Lemma}
\newtheorem{proposition}[theorem]{Proposition}
\newtheorem{definition}[theorem]{Definition}

\newtheorem{remark}[theorem]{Remark}
\newtheorem{corollary}[theorem]{Corollary}
\newtheorem{example}[theorem]{Example}

\def\to{\rightarrow}
\def\f{\mathfrak}
\def\c{\mathcal}
\def\b{\mathbf}
\def\r{\mathrm}
\def\bb{\mathbb}

\def\ot{\otimes}
\def\ov{\overline}
\date{}
\begin{document}
%%%%%%%%%%%%%%%%%%%%%%%%%%%%%%%%%%%%%%%%%%%%%%%%%%%%%%%%%%%%%%%%%%%%%%%%%%%%%%%%%%%%%%%%%%%%%%%%%%%%%%%%%%%%%%%%%%%%%%%%%%
\title{Path-Connected Components of Affine Schemes and\\ Algebraic K-Theory}
%%%%%%%%%%%%%%%%%%%%%%%%%%%%%%%%%%%%%%%%%%%%%%%%%%%%%%%%%%%%%%%%%%%%%%%%%%%%%%%%%%%%%%%%%%%%%%%%%%%%%%%%%%%%%%%%%%%%%%%%%%
\author{Maysam Maysami Sadr\thanks{sadr@iasbs.ac.ir}}
\affil{Department of Mathematics, Institute for Advanced Studies in Basic Sciences, Zanjan, Iran}
%%%%%%%%%%%%%%%%%%%%%%%%%%%%%%%%%%%%%%%%%%%%%%%%%%%%%%%%%%%%%%%%%%%%%%%%%%%%%%%%%%%%%%%%%%%%%%%%%%%%%%%%%%%%%%%%%%%%%%%%%%
\maketitle
%%%%%%%%%%%%%%%%%%%%%%%%%%%%%%%%%%%%%%%%%%%%%%%%%%%%%%%%%%%%%%%%%%%%%%%%%%%%%%%%%%%%%%%%%%%%%%%%%%%%%%%%%%%%%%%%%%%%%%%%%%
\begin{abstract}
We introduce a functor $\f{M}:\b{Alg}\times\b{Alg}^\r{op}\to\r{pro}\text{-}\b{Alg}$ constructed from representations of $\r{Hom}_\b{Alg}(A,B\ot?)$.
As applications, the following items are introduced and studied: (i) Analogue of the functor $\pi_0$ for algebras and affine schemes.
(ii) Cotype of Weibel's concept of strict homotopization. (iii) A homotopy invariant intrinsic singular cohomology theory
for affine schemes with cup product.
(iv) Some extensions of $\b{Alg}$ that are enriched over idempotent semigroups.
(v) Classifying homotopy pro-algebras for Corti\~{n}as-Thom's KK-groups and Weibel's homotopy K-groups.
\end{abstract}
%%%%%%%%%%%%%%%%%%%%%%%%%%%%%%%%%%%%%%%%%%%%%%%%%%%%%%%%%%%%%%%%%%%%%%%%%%%%%%%%%%%%%%%%%%%%%%%%%%%%%%%%%%%%%%%%%%%%%%%%%%
%%%%%%%%%%%%%%%%%%%%%%%%%%%%%%%%%%%%%%%%%%%%%%%%%%%%%%%%%%%%%%%%%%%%%%%%%%%%%%%%%%%%%%%%%%%%%%%%%%%%%%%%%%%%%%%%%%%%%%%%%%
%%%%%%%%%%%%%%%%%%%%%%%%%%%%%%%%%%%%%%%%%%%%%%%%%%%%%%%%%%%%%%%%%%%%%%%%%%%%%%%%%%%%%%%%%%%%%%%%%%%%%%%%%%%%%%%%%%%%%%%%%%
%%%%%%%%%%%%%%%%%%%%%%%%%%%%%%%%%%%%%%%%%%%%%%%%%%%%%%%%%%%%%%%%%%%%%%%%%%%%%%%%%%%%%%%%%%%%%%%%%%%%%%%%%%%%%%%%%%%%%%%%%%
%%%%%%%%%%%%%%%%%%%%%%%%%%%%%%%%%%%%%%%%%%%%%%%%%%%%%%%%%%%%%%%%%%%%%%%%%%%%%%%%%%%%%%%%%%%%%%%%%%%%%%%%%%%%%%%%%%%%%%%%%%
%%%%%%%%%%%%%%%%%%%%%%%%%%%%%%%%%%%%%%%%%%%%%%%%%%%%%%%%%%%%%%%%%%%%%%%%%%%%%%%%%%%%%%%%%%%%%%%%%%%%%%%%%%%%%%%%%%%%%%%%%%
\section{Introduction}\label{1910101626}
The main objective of this note is to introduce and study a specific bifunctor $$\f{M}:\b{Alg}\times\b{Alg}^\r{op}\to\r{pro}\text{-}\b{Alg}$$
on the category of (noncommutative and nonunital) algebras over any field $\bb{F}$. For any two algebras $A$ and $B$, the pro-algebra
$\f{M}_{A,B}$ is defined to be the pro-object of $\b{Alg}$ representing the set-valued functor $\r{Hom}_\b{Alg}(A,B\ot?)$.
We give an explicit construction of $\f{M}_{A,B}$ associated with each generator-set of $A$ and each vector-space basis of $B$.
It turns out that $\f{M}_{A,B}$ may be regarded as the algebra of polynomial functions on the noncommutative affine ind-scheme
of all morphisms $A\to B$. We consider some variants of $\f{M}$ related to unital, commutative, and reduced cases.
The behavior with respect to product and coproduct, and a specific \emph{exponential law} for the functor $\f{M}$ are considered.
It is also shown that $\f{M}$ preserves algebraic homotopy.
Using $\f{M}$ and regarding it as a (dual) pure-algebraic version of the ordinary mapping space functor on topological spaces
(see below), we introduce and study some new objects in homotopy theories of algebras and (noncommutative) affine schemes:

(i) We introduce a functor $\f{P}=\f{P}^\r{c}$ and some of its variants on $\b{Alg}$ that send any algebra $A$ to some specific subalgebra
$\f{P}(A)\subseteq A$. The associated functor on affine schemes turns out to be an algebraic version of the usual path-connected component functor
$\pi_0$ on topological spaces. Homotopy invariance, behavior with respect to direct sum and tensor product,
and some other properties of $\f{P}$ are considered.  Also, it is shown that in case
$\r{char}(\bb{F})=0$, $\f{P}$ is isomorphic to the de Rham cohomology-group
functor at degree $0$ on the category of unital commutative algebras.

(ii) For any functor $F$ on an (admissible) subcategory of $\b{Alg}$ with values in an arbitrary category we define the \emph{costrict} homotopization
of $F$ to be a homotopy invariant functor $\lceil F\rceil$, with the same domain and target as $F$, satisfying a cotype of the universal property of
Weibel's strict homotopization \cite{Weibel1}. It is shown that under some mild conditions on the domain and target categories,
the functor $\lceil F\rceil$ exists. It turns out that $\f{P}=\lceil\r{id}\rceil$. Also, it became clear that the
\emph{space} of \emph{path-connected components} of any affine scheme can be defined in three different fashions.

(iii) Using machinery of ordinary singular cohomology theory, with each functor $F$ on a specific subcategory of
$\b{Alg}$ with values in abelian groups we associate two homotopy invariant \emph{homology theories} $\ov{\c{H}}^*_F,\underline{\c{H}}^*_F$
for algebras. It is proved that $\ov{\c{H}}^0_F=\lceil F\rceil$ and thus the functors $\ov{\c{H}}^n_F$ may be regarded as higher ordered
costrict homotopizations of $F$. In case $F=\r{id}$ we have $\ov{\c{H}}^*_F=\underline{\c{H}}^*_F$, and we may consider the
cup product for the homology theory. The associated theory on affine schemes is called \emph{intrinsic singular cohomology theory}.
A de Rham Theorem at degree zero is proved for that.

(iv) We construct an extension $\bar{\b{Alg}}$ (and some of its variants) of $\b{Alg}$.
The objects of $\bar{\b{Alg}}$ are those of $\b{Alg}$ and the morphisms from an algebra $A$ to an algebra $B$ are
(various classes of) \emph{pro-ideals} of $\f{M}_{A,B}$. It is shown that $\bar{\b{Alg}}$ is an enriched category over partially ordered idempotent
semigroups, by inclusion and intersection of pro-ideals. Also, the K-group functor ${K_0}$ can be extended to a version of $\bar{\b{Alg}}$.

(v) We consider a bivariant \emph{homology theory} constructed as a pure-algebraic version of Cuntz's interpretation \cite{Cuntz1}
of the Kasparov bivariant K-theory of C*-algebras \cite{Kasparov1}. It is shown that for every unital algebra $C$
the canonical group-morphism from $K_0(C)$ into the Weibel homotopy K-group $KH_0(C)$ \cite{Weibel3} factors through $\f{QQ}(\bb{F},C)$.
Following a method introduced by Phillips \cite{Phillips2} we prove the existence of \emph{classifying homotopy pro-algebras}
for $\f{QQ}$-groups, Corti\~{n}as-Thom's KK-groups \cite{CortinasThom1}, and $KH_0$-groups. Thus, for instance, it is proved that
there exists a pro-object $\ov{KH}_0$ of the homotopy category $\r{Hot}(\b{Alg})$ with a cocommutative cogroup structure such that
$$KH_0(B)\cong[\ov{KH}_0,B]\hspace{10mm}(B\in\b{Alg}).$$

In the following we explain some simple ideas from classical Topology that are behind the definitions of $\f{M}$ and
the objects introduced in (i) and (iii).

Let $\b{Top}$ denote the category of topological spaces and $\b{Top}_\r{k}$ the subcategory of compact topological manifolds.
Consider the bifunctor $$\r{M}:\b{Top}_\r{k}^\r{op}\times\b{Top}_\r{k}\to\b{Top}$$
that associates with $(Y,X)$, the space of all (continuous) mappings from $Y$ into $X$, with compact-open topology.
It is an elementary fact that any map $\phi:Z\times Y\to X$ defines a $Z$-parameterized family $\{\psi(z)\}_z$ in $\r{M}_{Y,X}$ given by
\begin{equation*}\psi:Z\to\r{M}_{Y,X}\hspace{10mm}\psi(z)(y):=\phi(z,y).\end{equation*}
Identifying maps of the type $\phi$ with families of maps of the type $\psi$, we may define the space $\r{M}_{Y,X}$ as the parameter-space $M$
of a family $u:M\times Y\to X$ satisfying the following universal property: For any family $\phi:Z\times Y\to X$
with compact parameter space $Z$ there exists a unique map $\ov{\phi}:Z\to M$ making the following diagram commutative:
\begin{equation*}\xymatrix{M\times Y\ar[r]^-{u}&X\\Z\times Y\ar[u]^-{\ov{\phi}\times\r{id}_Y}\ar[-1,1]_{\phi}&}\end{equation*}
Dualizing the above categorical definition of $\r{M}_{Y,X}$, we get the definition of $\f{M}_{A,B}$: For any two algebras $A$ and $B$,
$\f{M}_{A,B}$ is defined to be an object of $\b{Alg}$ coming with a morphism $\Upsilon_{A,B}:A\to B\ot\f{M}_{A,B}$ and satisfying the following
universal property: For any algebra $C$ and every morphism $\phi:A\to B\ot C$ there exists a unique morphism $\ov{\phi}:\f{M}_{A,B}\to C$
making the following diagram commutative:
\begin{equation*}\xymatrix{B\ot\f{M}_{A,B}\ar[d]_-{\r{id}_B\ot\ov{\phi}}&A\ar[l]_-{\Upsilon_{A,B}}\ar[1,-1]^{\phi}\\B\ot C&}\end{equation*}
It will be shown that $\f{M}_{A,B}$ exists but not always as an algebra. Indeed, as in the general case $\r{M}_{Y,X}$
is a noncompact topological space, $\f{M}_{A,B}$ in general is a pro-object of $\b{Alg}$.
Using the mentioned universal property it is easily verified that the assignment
$(A,B)\to\f{M}_{A,B}$ defines a bifunctor (Theorem \ref{2011302233}).
Also, it is easily checked that there is a canonical bijection between $\bb{F}$-points
of the noncommutative affine ind-scheme $\c{S}(\f{M}_{A,B})$ associated to $\f{M}_{A,B}$ (i.e. the pro-morphisms $\f{M}_{A,B}\to\bb{F}$)
and the morphisms $A\to B$ (Corollary \ref{2012011947}). Thus, $\c{S}(\f{M}_{A,B})$ may be called noncommutative (or quantum) family of
\emph{all} the maps $\c{S}(B)\to\c{S}(A)$. For more details on the concept of quantum family see \cite{Sadr1}.

It is necessary to remark that the pro-algebras of the type $\f{M}_{A,B}$ and their properties have been known for a long time
(see for instance \cite{Gersten2}). But it seems that these objects have not been systematically studied.

Let $\pi_0:\b{Top}_\r{k}\to\b{Top}_\r{k}$ denote the functor that associates with any space the set of its
path-connected components endowed with quotient topology. For any $X\in\b{Top}_\r{k}$
we may identify $\c{C}(\pi_0X)$, the algebra of real-valued continuous functions on $\pi_0X$, with the
subalgebra of those functions $f\in\c{C}(X)$ with the property that for any continuous curve $\gamma:[0,1]\to X$,
$f\circ\gamma$ is constant. Thus if $X^{[0,1]}:=\r{M}([0,1],X)$ denotes the path-space of $X$ then $f\in\c{C}(X)$ belongs to $\c{C}(\pi_0X)$
iff the morphism $$\c{C}(X)\to\c{C}([0,1]\times X^{[0,1]})\hspace{5mm}g\mapsto[(t,\gamma)\mapsto g(\gamma(t))]$$ sends $f$ to a function of the form
$$1\ot\tilde{f}\in\c{C}([0,1])\ot\c{C}(X^{[0,1]})\subseteq\c{C}([0,1]\times X^{[0,1]})$$
for some $\tilde{f}\in\c{C}(X^{[0,1]})$.

Using the above formalism we may define the functor $\pi_0$ on the category of affine schemes:
Let $A$ be a unital commutative algebra regarded as the algebra of polynomial functions on the affine scheme $S=\c{S}(A)$ over $\bb{F}$.
Denote by $\f{M}^\r{cu},\Upsilon^\r{cu}$ the commutative unital versions of $\f{M},\Upsilon$ (Theorem \ref{2011302233}).
We let $\f{P}(A)$ be the subalgebra of those elements $a\in A$ such that the image of $a$ under the morphism
$$\Upsilon^\r{cu}_{A,\bb{F}[x]}:A\to\bb{F}[x]\ot\f{M}^\r{cu}_{A,\bb{F}[x]}$$
is of the form $1\ot\tilde{a}$ in each component of the pro-algebra $\bb{F}[x]\ot\f{M}^\r{cu}_{A,\bb{F}[x]}$ (Lemma \ref{2009301424}).
So, we have replaced the interval $[0,1]$ with the affine line $\bb{A}^1=\c{S}(\bb{F}[x])$. We define $$\pi_0S:=\c{S}(\f{P}(A)).$$

A continuous version of the ordinary singular cohomology theory with values in $\bb{R}$ may be simply defined as follows:
Let $\{\Delta_n\}_n$ denote the standard cosimplicial space. For any space $X$, the action of the cofunctor $\r{M}(?,X)$ on $\{\Delta_n\}_n$
gives rise to the simplicial space $\{\r{M}(\Delta_n,X)\}_n$ and then to the cosimplicial vector space $\{\c{C}(\r{M}(\Delta_n,X))\}_n$
that \emph{accidentally} is also a cosimplicial algebra. The cohomology groups of the Moore cochain complex associated with that cosimplicial
vector space may be called continuous singular cohomology groups of $X$ \cite{Mdzinarishvili1}.

A pure-algebraic version of the above formalism may be applied to affine schemes: Let $\bb{F}[\Delta]:=\{\bb{F}[\Delta_n]\}_n$ denote the standard
simplicial algebra \cite{CortinasThom1,Garkusha1} (see $\S$\ref{2009232252} for the definition). For any affine scheme $S=\c{S}(A)$
the action of $\f{M}_{A,?}$ on $\bb{F}[\Delta]$ gives rise to the cosimplicial pro-algebra $\f{M}_{A,\bb{F}[\Delta]}$. The cohomology
groups of the Moore complex of $\underleftarrow{\lim}\f{M}_{A,\bb{F}[\Delta]}$ are called intrinsic singular cohomology groups of $S$.

In the rest of this section we fix our notations. In $\S$ \ref{1909142202} we construct the functor $\f{M}$ and its variants, and consider
some basic properties of $\f{M}$. In $\S$ \ref{1910191623} we review the classical notion of algebraic homotopy for pro- and ind-morphisms,
and show that $\f{M}$ is homotopy preserving. In $\S$ \ref{1909241128}-\ref{2011161438} we consider respectively the items (i)-(v).
%%%%%%%%%%%%%%%%%%%%%%%%%%%%%%%%%%%%%%%%%%%%%%%%%%%%%%%%%%%%%%%%%%%%%%%%%%%%%%%%%%%%%%%%%%%%%%%%%%%%%%%%%%%%%%%%%%%%%
\subsection*{Notations \& Conventions}
We denote by $\b{Set},\b{Ab},\b{Chain},\b{CChain}$ respectively the categories of sets, abelian groups, and chain and cochain complexes of abelian
groups. Let $\b{C}$ be a category. The category of simplicial (resp. cosimplicial) objects of $\b{C}$ is denoted by $\r{sim}\text{-}\b{C}$
(resp. $\r{cosim}\text{-}\b{C}$). The category of pro-objects (resp. ind-object) of $\b{C}$ (\cite[Appendix]{ArtinMazur1}) is denoted by
$\r{pro}\text{-}\b{C}$ (resp. $\r{ind}\text{-}\b{C}$). An object of $\r{pro}\text{-}\b{C}$ (alternatively called a pro-object of $\b{C}$)
is an indexed family $\{C_i\}_{i\in I}$ of objects of $C$ over a directed set $I$ together with the structural morphisms $\alpha_{ii'}:C_{i'}\to C_i$
for $i'\geq i$ which are compatible: $\alpha_{ii}=\r{id}_{C_i}$ and $\alpha_{ii''}=\alpha_{ii'}\alpha_{i'i''}$. For pro-objects
$C=(I,C_i,\alpha_{ii'})$ and $D=(J,D_j,\beta_{jj'})$ of $\b{C}$ it is defined that
$$\r{Hom}_{\r{pro}\text{-}\b{C}}(C,D):=\lim_{\infty\leftarrow j}\lim_{i\to\infty}\r{Hom}_\b{C}(C_i,D_j).$$
(The structure of $\r{Hom}_{\r{pro}\text{-}\b{C}}(C,D)$ may be explained as follows:
A represented pro-morphism from $C$ to $D$ is distinguished by a function $f:J\to I$ and a family $\{\phi_{j}:C_{f(j)}\to D_j\}_j$
of morphisms with the property that if $j'\geq j$ then there exists $i\geq f(j),f(j')$ such that
$\phi_j\alpha_{f(j)i}=\beta_{jj'}\phi_{j'}\alpha_{f(j')i}$. Two represented pro-morphisms $(f,\phi_j)$ and $(g,\psi_j)$ are equivalent
if for every $j$ there exists $i\geq f(j),g(j)$ such that $\phi_j\alpha_{f(j)i}=\psi_j\alpha_{g(j)i}$. Then, $\r{Hom}_{\r{pro}\text{-}\b{C}}(C,D)$
may be identified with the set of equivalence classes of represented pro-morphisms.)
We always collapse pro-pro- and ind-ind-objects to pro- and ind-objects, without further comment. We use freely the canonical embeddings
$$\b{C}\subset\r{pro}\text{-}\b{C}\hspace{3mm}\text{and}\hspace{3mm}\b{C}\subset\r{ind}\text{-}\b{C}.$$
We also freely use the canonical extensions of any functor $\b{C}\to\b{D}$ to the functors
$$\r{pro}\text{-}\b{C}\to\r{pro}\text{-}\b{D},\hspace{2mm}\r{ind}\text{-}\b{C}\to\r{ind}\text{-}\b{D},
\hspace{2mm}\r{sim}\text{-}\b{C}\to\r{sim}\text{-}\b{D}.$$
Inverse limit $\underleftarrow{\lim}$ (if exists) is considered as a functor from $\r{pro}\text{-}\b{C}$ to $\b{C}$.
Recall that a functor $F:\b{C}\to\b{Set}$ is said to be pro-representable if there exist a pro-object $C=\{C_i\}$
of $\b{C}$  and a natural isomorphism $\Phi:\r{Hom}_{\r{pro}\text{-}\b{C}}(C,?)\to F$.
Note that $\Phi$ is exactly distinguished by the family $\{\phi_i\in F(C_i)\}$ where $\phi_i:=\Phi(\r{id}_{C_i})$.
We call $$(\{C_i\}_i,\{\phi_i\}_i)$$ a pro-representation for $F$, and say that $F$ is pro-represented by $C$.\\
Throughout, we work over a fixed field $\bb{F}$; all vector spaces and algebras are understood over $\bb{F}$.
$\b{Vec}$ denotes category of vector spaces. The symbol $\ot$ without any subscript denotes $\ot_\bb{F}$. The sub- and super-scripts
$\r{nc,c,u,r,fg,fp}$ stand for noncommutative, commutative, unital, reduced, finitely generated, finitely presented. From now on,
the category of all algebras is denoted by $\b{A}_\r{nc}$. The category of unital algebras and unit preserving morphisms is denoted by
$\b{A}_\r{u}$. The category $\b{A}_*$ is similarly considered for $*=\r{c},\r{cu},\r{fg},\r{fp},\ldots$.
Thus, for instance, $\b{A}_\r{cufg}$ denotes the category of finitely generated commutative unital algebras and unit preserving morphisms.
Tensor product, product, and coproduct of pro- and ind-algebras are defined componentwise. For any pro-algebra $A$ by a point of $A$ we mean a
pro-morphism from $A$ into $\bb{F}$. The set of all points and all nonzero points of $A$ are denoted by $\r{Pnt}(A)$ and $\r{Pnt}_{\neq0}(A)$. The
category of affine schemes (over $\bb{F}$) as usual is defined by $\b{Aff}:=\b{A}_\r{cu}^\r{op}$. The affine scheme associated with
$A\in\b{A}_\r{cu}$ is denoted by $\c{S}(A)$. $\r{M}_n$ denotes the algebra of $n\times n$ matrixes with entries in $\bb{F}$.
We denote with the same symbol the functor $\b{A}_\r{nc}\to\b{A}_\r{nc}$ given by $A\mapsto\r{M}_n(A):=\r{M}_n\ot A$.
%%%%%%%%%%%%%%%%%%%%%%%%%%%%%%%%%%%%%%%%%%%%%%%%%%%%%%%%%%%%%%%%%%%%%%%%%%%%%%%%%%%%%%%%%%%%%%%%%%%%%%%%%%%%%%%%%%%%%
%%%%%%%%%%%%%%%%%%%%%%%%%%%%%%%%%%%%%%%%%%%%%%%%%%%%%%%%%%%%%%%%%%%%%%%%%%%%%%%%%%%%%%%%%%%%%%%%%%%%%%%%%%%%%%%%%%%%%
\subsection*{Acknowledgement}
The author would like to express his sincere gratitude to Professor Corti\~{n}as for many valuable comments, hints, and remarks
on the early version of this manuscript.
%%%%%%%%%%%%%%%%%%%%%%%%%%%%%%%%%%%%%%%%%%%%%%%%%%%%%%%%%%%%%%%%%%%%%%%%%%%%%%%%%%%%%%%%%%%%%%%%%%%%%%%%%%%%%%%%%%%%%%%%%%
%%%%%%%%%%%%%%%%%%%%%%%%%%%%%%%%%%%%%%%%%%%%%%%%%%%%%%%%%%%%%%%%%%%%%%%%%%%%%%%%%%%%%%%%%%%%%%%%%%%%%%%%%%%%%%%%%%%%%%%%%%
%%%%%%%%%%%%%%%%%%%%%%%%%%%%%%%%%%%%%%%%%%%%%%%%%%%%%%%%%%%%%%%%%%%%%%%%%%%%%%%%%%%%%%%%%%%%%%%%%%%%%%%%%%%%%%%%%%%%%%%%%%
%%%%%%%%%%%%%%%%%%%%%%%%%%%%%%%%%%%%%%%%%%%%%%%%%%%%%%%%%%%%%%%%%%%%%%%%%%%%%%%%%%%%%%%%%%%%%%%%%%%%%%%%%%%%%%%%%%%%%%%%%%
\section{The Main Definition}\label{1909142202}
Let $\b{B},\b{C},\b{D}$ be arbitrary categories and let $F:\b{C}\times\b{D}\to\b{B}$ be a functor. For objects $B\in\b{B}$ and $C\in\b{C}$
consider the functor\begin{equation*}%\label{2008202146}
F_{B,C}:=\r{Hom}_\b{B}(B,F(C,?)):\b{D}\to\b{Set}.\end{equation*}Let $D$ be a pro-object of $\b{D}$ and $\phi:B\to F(C,D)$ be a pro-morphism.
\begin{lemma}\label{2008210701}
\emph{$(D,\phi)$ is a pro-representation for $F_{B,C}$ iff the following universal property holds:
For every pro-object $E$ of $\b{D}$ and every pro-morphism $\psi:B\to F(C,E)$ there exists a unique pro-morphism $\bar{\psi}:D\to E$
such that $\psi=F(\r{id}_C,\bar{\psi})\phi$.}\end{lemma}
The functor $F$ is said to be $\r{pro}\text{-}\b{D}$-representable if for every $B,C$, $F_{B,C}$ is pro-representable. For $D$ and $\phi$ as above,
we use respectively the notations $$\f{M}(B,C)=\f{M}_{B,C}\hspace{3mm}\text{and}\hspace{3mm}\Upsilon(B,C)=\Upsilon_{B,C}.$$
Let $f:B\to B'$ and $g:C'\to C$ be morphisms respectively in $\b{B}$ and $\b{C}$. The universal property of $(\f{M}_{B,C},\Upsilon_{B,C})$
shows that there is a unique pro-morphism $$\f{M}(f,g)=\f{M}_{f,g}:\f{M}_{B,C}\to\f{M}_{B',C'}$$ satisfying the identity
$$F(g,\r{id}_{\f{M}_{B',C'}})\Upsilon_{B',C'}f=F(\r{id}_C,\f{M}_{f,g})\Upsilon_{B,C}.$$
It is clear that $\f{M}(\r{id}_B,\r{id}_C)=\r{id}_{\f{M}(B,C)}$. For morphisms $f':B'\to B'',g':C''\to C'$, again universal
property of the involved objects shows that $$\f{M}(f'f,gg')=\f{M}(f',g')\f{M}(f,g).$$
Ignoring some set-theoretical difficulties about \emph{the choice} of a pro-representation $(\f{M}_{B,C},\Upsilon_{B,C})$ for $F_{B,C}$,
we have the following easily checked result.
\begin{proposition}\label{2008210703}
\emph{Suppose that the functor $$F:\b{C}\times\b{D}\to\b{B}$$ is $\r{pro}\text{-}\b{D}$-representable. Then, $\f{M}$ may be considered as a functor
$$\f{M}:\b{B}\times\b{C}^\r{op}\to\r{pro}\text{-}\b{D},$$ and $\Upsilon$ as a natural transformation, in the obvious way. For any $C\in\b{C}$,
the functor$$\f{M}(?,C):\r{pro}\text{-}\b{B}\to\r{pro}\text{-}\b{D}$$
is left-adjoint to the functor $$F(C,?):\r{pro}\text{-}\b{D}\to\r{pro}\text{-}\b{B}.$$
Thus, for $B\in\r{pro}\text{-}\b{B}$, $C\in\b{C}$, and $D\in\r{pro}\text{-}\b{D}$, we have the following natural bijection:
$$\r{Hom}_{\r{pro}\text{-}\b{D}}(\f{M}(B,C),D)\cong\r{Hom}_{\r{pro}\text{-}\b{B}}(B,F(C,D))).$$
For $B\in\b{B}$, $C\in\r{ind}\text{-}\b{C}$, and $D\in\b{D}$, we have the following bijection:
$$\r{Hom}_{\r{pro}\text{-}\b{D}}(\f{M}(B,C),D)\cong\r{Hom}_{\r{ind}\text{-}\b{B}}(B,F(C,D)).$$}
\end{proposition}
We now turn to our main construction in the category of algebras.\\
Let $A,B$ be algebras. Suppose $\theta=\{\delta_a\}_{a\in G}$ is a family of finite linearly independent subsets $\delta_a$ of $B$ indexed by a set
$G\subseteq A$ that generates $A$ as an algebra. Denote by $\f{M}\theta$ the universal algebra in $\b{A}_\r{nc}$
generated by symbols $z_{a,v}$, where $a\in G$ and $v\in\delta_a$, subject to the condition that the assignment
\begin{equation}\label{2009191935}
a\mapsto\sum_{v\in\delta_a}v\ot z_{a,v}\end{equation}defines a morphism ${\Upsilon}\theta:A\to B\ot\f{M}\theta$ in $\b{A}_\r{nc}$.
The pair $(\f{M}\theta,\Upsilon\theta)$ has the following universal property: For every $C\in\b{A}_\r{nc}$ and any morphism
$\varphi:A\to B\ot C$ in $\b{A}_\r{nc}$ with the property that for every $a\in G$, $\varphi(a)$ is a linear combination of vectors in $\delta_a$,
there is a unique morphism $\ov{\varphi}:\f{M}\theta\to C$ in $\b{A}_\r{nc}$ such that $\varphi=(\r{id}_B\ot\ov{\varphi}){\Upsilon}\theta$. We call
$\ov{\varphi}$ the morphism $\theta$-associated with $\varphi$.

Now, fix a vector basis $V$ for the underlying vector space of $B$, and let $\Theta=\Theta_{G,V}$ denote the set of all families
$\theta=\{\delta_a\}_{a\in G}$ as above with $\delta_a\subseteq V$. For $\theta,\theta'\in\Theta$, write $\theta\subseteq\theta'$ if
$\delta_a\subseteq\delta_a'$ for every $a\in G$. Then, $\subseteq$ makes $\Theta$ into a directed set. For $\theta\subseteq\theta'$ let
$\phi_{\theta'\theta}:\f{M}\theta'\to\f{M}\theta$ denote the morphism $\theta'$-associated with $\Upsilon\theta$. We have
$\phi_{\theta\theta}=\r{id}$ and if $\theta\subseteq\theta'\subseteq\theta''$ then
$\phi_{\theta''\theta}=\phi_{\theta'\theta}\phi_{\theta''\theta'}$. So, the data $\{\f{M}\theta,\phi_{\theta'\theta}\}$
distinguish a pro-algebra $\f{M}^\r{nc}_{A,B}$ indexed by $\Theta$ and the family $\{{\Upsilon}\theta\}$ defines a pro-morphism
$$\Upsilon^\r{nc}_{A,B}:A\to B\ot\f{M}^\r{nc}_{A,B}.$$ We have the following universal property. Let $C=\{C_i\}_i$ be a pro-algebra and let
$$\varphi:A\to B\ot C\hspace{10mm}\varphi=\{\varphi_i:A\to B\ot C_i\}_i$$ denote a pro-morphism.
For every $i$, there exists $\theta=\{\delta_a\}$ in $\Theta$ such that $\varphi_i(a)$ a linear combinations of vectors in $\delta_a$.
Let $\ov{\varphi_i}:\f{M}\theta\to C_i$ be the morphism $\theta$-associated with $\varphi_i$.
Then, the family $\{\ov{\varphi_i}\}_i$ defines a unique pro-morphism $$\ov{\varphi}:\f{M}^\r{nc}_{A,B}\to C\hspace{5mm}\varphi=\{\ov{\varphi_i}\}_i$$
satisfying $$\varphi=(\r{id}_B\ot\ov{\varphi})\Upsilon^\r{nc}_{A,B}.$$ We call $\ov{\varphi}$ the pro-morphism associated with $\varphi$.
By Lemma \ref{2008210701}, the pair\begin{equation}\label{2011301249}(\f{M}^\r{nc}_{A,B},\Upsilon^\r{nc}_{A,B})\end{equation} is a pro-representation
for the functor $$\r{Hom}_{\b{A}_\r{nc}}(A,B\ot?):\b{A}_\r{nc}\to\b{Set}.$$ Thus, tensor product of algebras as the functor
\begin{equation*}%\label{2008210702}
\ot:\b{A}_\r{nc}\times\b{A}_\r{nc}\to\b{A}_\r{nc}, \end{equation*} is $\r{pro}\text{-}\b{A}_\r{nc}$-representable.
Similarly, it can be shown that the functor $$\ot:\b{A}_\r{u}\times\b{A}_\r{u}\to\b{A}_\r{u}$$ is $\r{pro}\text{-}\b{A}_\r{u}$-representable.
Also, for $*\in\{\r{c,cr}\}$ ($*\in\{\r{cu,cur}\}$) the functors
$$\ot:\b{A}_\r{nc}\times\b{A}_*\to\b{A}_\r{nc}\hspace{10mm}(\ot:\b{A}_\r{u}\times\b{A}_*\to\b{A}_\r{u})$$
are $\r{pro}\text{-}\b{A}_*$-representable. We denote by $(\f{M}^*_{A,B},\Upsilon^*_{A,B})$ the pro-representations for these functors,
constructed similar to the pair (\ref{2011301249}). Thus, for instance, for unital algebras $A,B$, $\f{M}^\r{cur}_{A,B}$ is an object in
$\r{pro}\text{-}\b{A}_\r{cur}$ given by the inverse system $\{\f{M}\theta,\phi_{\theta'\theta}\}$ indexed over $\Theta$ as above,
where $\f{M}\theta$ is the universal algebra in $\b{A}_\r{cur}$ generated by symbols $z_{a,v}$ subject to the condition that the assignment
(\ref{2009191935}) defines a unit-preserving morphism  from $A$ into $B\ot\f{M}\theta$.
The following is a special case of Proposition \ref{2008210703}.
\begin{theorem}\label{2011302233}
\emph{For $*\in\{\r{nc,c,cr}\}$, $\f{M}^*$ may be considered as a functor $$\f{M}^*:\b{A}_\r{nc}\times\b{A}_\r{nc}^\r{op}\to\r{pro}\text{-}\b{A}_*$$
and $\Upsilon^*$ as a natural transformation, in the obvious way. For any $B\in\b{A}_\r{nc}$,
the functor$$\f{M}^*(?,B):\r{pro}\text{-}\b{A}_\r{nc}\to\r{pro}\text{-}\b{A}_*$$
is left-adjoint to the functor $$B\ot?:\r{pro}\text{-}\b{A}_*\to\r{pro}\text{-}\b{A}_\r{nc}.$$
For $A\in\r{pro}\text{-}\b{A}_\r{nc}$, $B\in\b{A}_\r{nc}$, and $C\in\r{pro}\text{-}\b{A}_*$, we have the natural bijection
$$\r{Hom}_{\r{pro}\text{-}\b{A}_*}(\f{M}^*(A,B),C)\cong\r{Hom}_{\r{pro}\text{-}\b{A}_\r{nc}}(A,B\ot C),$$
and for $A\in\b{A}_\r{nc}$, $B\in\r{ind}\text{-}\b{A}_\r{nc}$, and $C\in\b{A}_*$, the bijection $$\r{Hom}_{\r{pro}\text{-}\b{A}_*}(\f{M}^*(A,B),C)\cong\r{Hom}_{\r{ind}\text{-}\b{A}_\r{nc}}(A,B\ot C).$$
The above statements hold also for $*\in\{\r{u,cu,cur}\}$ if $\b{A}_\r{nc}$ is replaced by $\b{A}_\r{u}$.}
\end{theorem}
\begin{corollary}\label{2012011947}
\emph{For unital algebras $A,B$, we have the following identification:
$$\r{Hom}_{\b{A}_\r{u}}(A,B)\cong\r{Pnt}_{\neq0}(\f{M}^*_{A,B})\hspace{10mm}(*\in\{\r{u,cu,cur}\}).$$
The similar statement holds for nonunital case.}
\end{corollary}
\begin{proof}
It follows from Theorem \ref{2011302233} by putting $C=\bb{F}$.
\end{proof}
Using universality of $\f{M}^*$ it is easily verified that
$$\f{M}^\r{c}=(\f{M}^\r{nc})_\r{com},\hspace{2mm}\f{M}^\r{cu}=(\f{M}^\r{u})_\r{com},
\hspace{2mm}\f{M}^\r{cr}=(\f{M}^\r{c})_\r{red},\hspace{2mm}\f{M}^\r{cur}=(\f{M}^\r{cu})_\r{red},$$
where $A\mapsto A_\r{com}$ ($\b{A}_\r{nc}\to\b{A}_\r{c}$) and  $A\mapsto A_\r{red}$ ($\b{A}_\r{c}\to\b{A}_\r{cr}$) denote
commutativization and reduction functors. We have also canonical natural transformations
\begin{equation}\label{2011111033}
\f{M}^\r{nc}\to\f{M}^\r{c}\to\f{M}^\r{cr}\hspace{3mm}\text{and}\hspace{3mm}\f{M}^\r{u}\to\f{M}^\r{cu}\to\f{M}^\r{cur}.
\end{equation}
For the relation between unital and nonunital cases we only mention the following result.
Let $+:\b{A}_\r{nc}\to\b{A}_\r{u}$ denote the unitization functor. It is well-known that for any algebra $A$ and any unital algebra $B$,
every morphism $f:A\to B$ extends uniquely to a unit-preserving morphism $A^+\to B$.
(Indeed, this fact says that $+$ is a left-adjoint to the embedding $\b{A}_\r{u}\to\b{A}_\r{nc}$.)
By abuse of notation, we denote this extension of $f$ by $f^+$. The same convention is applied to pro-morphisms.
\begin{theorem}\label{2012091919}
\emph{Let $A\in\b{A}_\r{nc},B\in\b{A}_\r{u}$. Then there is a canonical isomorphism
$$(\f{M}^\r{nc}_{A,B})^+\cong\f{M}^\r{u}_{A^+,B},\hspace{3mm}(\f{M}^\r{c}_{A,B})^+\cong\f{M}^\r{cu}_{A^+,B},
\hspace{3mm}(\f{M}^\r{cr}_{A,B})^+\cong\f{M}^\r{cur}_{A^+,B}.$$}
\end{theorem}
\begin{proof}
Let $f:=(\r{id}_B\ot\r{e})\Upsilon^\r{nc}_{A,B}$ where $\r{e}:\f{M}^\r{nc}_{A,B}\to\f{M}_{A,B}^{\r{nc}+}$ denotes the canonical embedding.
We have $f^+:A^+\to B\ot\f{M}^{\r{nc}+}_{A,B}$ in $\r{pro}\text{-}\b{A}_\r{u}$. Let $\ov{f^+}:\f{M}^\r{u}_{A^+,B}\to\f{M}^{\r{nc}+}_{A,B}$
be the pro-morphism associated with $f^+$. Thus $f^+=(\r{id}\ot\ov{f^+})\Upsilon^\r{u}_{A^+,B}$. Let $g:=\Upsilon^\r{u}_{A^+,B}\r{e}$
where $\r{e}$ this time shows the canonical embedding $A\to A^+$. Let $\ov{g}:\f{M}^\r{nc}_{A,B}\to\f{M}^\r{u}_{A^+,B}$ denote
the pro-morphism associated with $g$. Thus $g=(\r{id}_B\ot\ov{g})\Upsilon^\r{nc}_{A,B}$. We have
$\ov{g}^+:\f{M}^{\r{nc}+}_{A,B}\to\f{M}^\r{u}_{A^+,B}$. It follows from universality of $\f{M}^\r{nc}$ and $\f{M}^\r{u}$ that $\ov{f^+}$
and $\ov{g}^+$ are inverses of each other. The proof is complete.
\end{proof}
From now on, when there is no confusion about $*\in\{\r{nc,c,cr,u,cu,cur}\}$, we will use the short
notations $\f{M},\Upsilon$ instead of $\f{M}^*,\Upsilon^*$. In the following we consider some basic properties of $\f{M}$.
\begin{theorem}\label{2009080009}
\emph{\begin{enumerate}
\item[(i)] $\f{M}^\r{nc}_{A,\bb{F}}\cong A$ for $A\in\b{A}_\r{nc}$, and $\f{M}^\r{u}_{\bb{F},B}\cong\bb{F}$ for $B\in\b{A}_\r{u}$.
\item[(ii)] For any $A\in\b{A}_\r{nc}$ there is a canonical bijection $\r{Pnt}(\f{M}^\r{nc}_{\bb{F},A})\cong\{a\in A:a^2=a\}$.
Similarly, for every $n>1$ there is a canonical bijection between $\r{Pnt}(\f{M}^\r{nc}_{\bb{F}^n,A})$
and the set of $n$-tuples of pairwise orthogonal idempotents of $A$.
\item[(iii)] For $A\in\b{A}_\r{nc}$ and $n>1$, $\f{M}^\r{nc}_{A,\bb{F}^n}$ is a coproduct of $n$ copies of $A$ in $\b{A}_\r{nc}$.
\item[(iv)] If $A\in\b{A}_\r{ncfg},A'\in\b{A}_\r{ncfp},B\in\b{A}_\r{nc}$ then
$\f{M}^\r{nc}_{A,B}\in\r{pro}\text{-}\b{A}_\r{ncfg},\f{M}^\r{nc}_{A',B}\in\r{pro}\text{-}\b{A}_\r{ncfp}$.
\item[(v)] If $B$ is a finite-dimensional algebra then
$\f{M}^\r{nc}_{A,B}\cong\underleftarrow{\lim}\f{M}^\r{nc}_{A,B}$ for any algebra $A$.
\end{enumerate}}
\end{theorem}
\begin{proof}
Straightforward.
\end{proof}
\begin{example}
\emph{$\f{M}^\r{nc}_{\bb{F},\c{M}_n}$ is the universal algebra generated by the symbols $\{z_{ij}\}_{i,j=1}^n$ satisfying
$z_{ij}=\sum_{k=1}^nz_{ik}z_{kj}$. We have $\Upsilon^\r{nc}_{\bb{F},\c{M}_n}(1)=\sum_{ij=1}^ne_{ij}\ot z_{ij}$ where $\{e_{ij}\}_{i,j=1}^n$ denotes
the canonical vector basis of $\c{M}_n$.}
\end{example}
\begin{example}
\emph{$\f{M}^\r{u}(\bb{F}[x],\bb{F}^n)$ is isomorphic to $\bb{F}_\r{nc}[x_1,\ldots,x_n]$, the algebra of polynomials in noncommutating
indeterminates $x_1,\ldots,x_n$.}
\end{example}
\begin{example}\label{2010021333}
\emph{Let $A\in\b{A}_\r{ncfg},B\in\b{A}_\r{nc}$. Suppose $G$ is a finite generator-set for $A$, and $V$ is a vector basis for $B$.
For every finite subset $S$ of $V$, let $\theta_S\in\Theta_{G,V}$ denote the family $\{\delta_a\}_a$ where $\delta_a=S$ for every $a\in G$.
Then $\{\theta_S\}_S$ is a cofinal subset of $\Theta_{G,V}$, and hence the pro-algebra $\f{M}^\r{nc}_{A,B}$ is described also by the inverse system
$\{\f{M}\theta_S\}_S$ of algebras. Moreover, suppose that $V$ is countable and $V=\{v_1,v_2,\ldots\}$, and let $S_n:=\{v_1,\ldots,v_n\}$.
Then, $\f{M}_{A,B}$ is also described by the following \emph{nice} inverse subsystem of $\{\f{M}\theta_S\}_S$:
$$\f{M}\theta_{S_1}\leftarrow\cdots\leftarrow\f{M}\theta_{S_{n}}\leftarrow\f{M}\theta_{S_{n+1}}\leftarrow\cdots.$$
Here, $\f{M}\theta_{S_n}$ is generated by $\{z_{a,i}:a\in G,i=1,\ldots,n\}$, and
$\f{M}^\r{nc}\theta_{S_{n+1}}\to\f{M}^\r{nc}\theta_{S_{n}}$ is given by $z_{a,i}\mapsto z_{a,i}$ for $i=1,\ldots,n$, and $z_{a,n+1}\mapsto0$.}
\end{example}
Let $C\in\b{A}_\r{cufg}$, and let $G=\{g_1,\ldots,g_n\}$ be a finite ordered generator-set for $C$.
Let $\phi$ be the morphism $\bb{F}[x_1,\ldots,x_n]\to C$
given by $x_i\mapsto g_i$. Denote by $\c{Z}(C,G)$ the zero-locus of $\ker(\phi)$, i.e. the algebraic set in $\bb{A}^n$ of $n$-tuples
$(\lambda_1,\ldots,\lambda_n)$ under which the evaluation of every polynomial in $\ker(\phi)$, is zero. It is clear that the assignment
$$(\lambda_1,\ldots,\lambda_n)\mapsto[g_1\mapsto\lambda_1,\ldots,g_n\mapsto\lambda_n]$$ defines a bijection from $\c{Z}(C,G)$ onto
$\r{Pnt}_{\neq0}(C)$. Thus, the choice of $G$ gives rise to the structure $\c{Z}(C,G)$ of an algebraic set on $\r{Pnt}_{\neq0}(C)$.
Similarly, if $C=(C_i)_{i}$ be in $\r{pro}\text{-}\b{A}_\r{cufg}$ then any family $G=\{G_i\}_i$ where $G_i$ is a finite ordered generator-set for
$A_i$, gives rise to an algebraic ind-set structure $\c{Z}(C,G):=\{\c{Z}(C_i,G_i)\}_i$ on the set
$$\r{Pnt}_{\neq0}(C)=\lim_{i\to\infty}\r{Pnt}_{\neq0}(C_i).$$
(We say that a set $X$ has the structure of a ind-set $(X_i)_i$ if we have a distinguished bijection between $X$ and $\lim_{i\to\infty}X_i$,
and for every $i$, $X_i\to\lim_{i\to\infty}X_i$ is injective.)
\begin{theorem}
\emph{Let $A\in\b{A}_\r{ufg}$ and let $B\in\b{A}_\r{u}$. Every finite ordered generator-set for $A$ and any ordered vector basis for $B$, give rise
to a canonical algebraic ind-set structure on $\r{Hom}_{\b{A}_\r{u}}(A,B)$.}
\end{theorem}
\begin{proof}
Consider the pro-algebra $\f{M}^\r{cu}_{A,B}$ with the model $\{\f{M}\theta_S\}_S$ described in Example \ref{2010021333}.
$\f{M}\theta_S$ is generated by elements $z_{a,v}$ indexed by the set $G\times S$. This set may be endowed the dictionary ordering
induced by orderings on $G$ and $V$. Then, the method described above gives rise to an algebraic ind-set structure on $\r{Pnt}_{\neq0}(\f{M}^\r{cu}_{A,B})$. Then, the desired result follows from Corollary \ref{2012011947}.
\end{proof}
For an application of the structure given by the above theorem, see \cite{KambayashiMiyanishi1}.
\begin{theorem}\label{1910101520}
\emph{Suppose $\star$ denotes coproduct in $\b{A}_\r{nc}$. There is an isomorphism:
$$\f{M}^\r{nc}(A,B\oplus B')\cong\f{M}^\r{nc}(A,B)\star\f{M}^\r{nc}(A,B')\hspace{10mm}(A,B,B'\in\b{A}_\r{nc})$$}
\end{theorem}\begin{proof}Let $A$ be generated by $G$, and $V\subset B,V'\subset B'$ be vector bases. For $\theta=\{\delta_a\}\in\Theta_{G,V}$ and
$\theta'=\{\delta'_a\}\in\Theta_{G,V'}$, let $\theta\cup\theta'\in\Theta_{G,V\cup V'}$ denote the family $\{\delta_a\cup\delta'_a\}$ where $V\cup V'$
is considered as a vector basis for $B\oplus B'$. It is enough to prove that $$\f{M}(\theta\cup\theta')=\f{M}(\theta)\star\f{M}(\theta').$$
We do that by checking that $\f{M}(\theta\cup\theta')$ has the required universal property. Let
$\phi_\theta:=(\r{p}\ot\r{id})\Upsilon(\theta\cup\theta')$ where $\r{p}:B\oplus B'\to B$ is the canonical projection. Then,
$\ov{\phi}_\theta:\f{M}\theta\to\f{M}(\theta\cup\theta')$, the morphism $\theta$-associated with $\phi_\theta$, plays the role of coproduct
structural morphism. Similarly, the structural morphism $\ov{\phi'}_{\theta'}:\f{M}\theta'\to\f{M}(\theta\cup\theta')$ is defined. Suppose that
$\psi:\f{M}\theta\to C,\psi':\f{M}\theta'\to C$ are arbitrary morphisms, and let $$\varphi=[(\r{id}_B\ot\psi)\Upsilon\theta]\oplus[(\r{id}_{B'}\ot\psi)\Upsilon\theta'].$$
Then, $\ov{\varphi}_{\theta\cup\theta'}:\f{M}(\theta\cup\theta')\to C$, the morphism $(\theta\cup\theta')$-associated with $\varphi$, is the unique
morphism satisfying $\psi=\ov{\varphi}_{\theta\cup\theta'}\ov{\phi}_{\theta}$ and $\psi'=\ov{\varphi}_{\theta\cup\theta'}\ov{\phi'}_{\theta'}$.
The proof is complete.\end{proof}
\begin{theorem}\label{1910111415}
\emph{(Exponential Law) There is a canonical isomorphism:
$$\f{M}^\r{nc}(A,B\ot B')\cong\f{M}^\r{nc}(\f{M}^\r{nc}(A,B),B')\hspace{10mm}(A,B,B'\in\b{A}_\r{nc}).$$}
\end{theorem}\begin{proof}
Suppose $G\subseteq A$ is a generator, and $V\subset B,V'\subset B'$ are vector bases. For $S\subseteq V,S'\subseteq V'$, denote by $S\ot S'$
the set $\{v\ot v':v\in S,v'\in S'\}$. For $\theta=\{\delta_a\}$ and $\theta'=\{\delta'_a\}$ respectively in $\Theta_{G,V}$ and $\Theta_{G,V'}$,
let $\theta\ot\theta'\in\Theta_{G,V\ot V'}$ denote $\{\delta_a\ot\delta'_a\}$. We know that $\f{M}\theta$ is generated by $G\times\theta:=\{z_{a,v}:a\in G,v\in\delta_a\}$. Let $\theta'|\theta$ in $\Theta_{G\times\theta,V'}$ denote $\{\eta_{z_{a,v}}\}$ where
$\eta_{z_{a,v}}:=\delta'_a$. The class $\{\theta\ot\theta'\}$ is cofinal with $\Theta_{G,V\ot V'}$, and the class $\{\theta'|\theta\}$ is
cofinal with $\Theta_{G\times\theta,V'}$. Thus, the pro-algebras $\f{M}_{A,B\ot B'}$ and $\f{M}_{\f{M}\theta,B'}$ are described respectively by
inverse systems $\{\f{M}(\theta\ot\theta')\}$ and $\{\f{M}(\theta'|\theta)\}$. Now, the desired result follows from the canonical isomorphism
of $\f{M}(\theta\ot\theta')$ and $\f{M}(\theta'|\theta)$ which can be proved by using similar arguments as in \cite[Theorem 2.10]{Sadr1}.
\end{proof}The proof of the following result is similar to that of \cite[Theorem 2.8]{Sadr1} and omitted.
\begin{theorem}
\emph{There exists a canonical isomorphism:
$$\f{M}^\r{cu}(A\ot A',B)\cong\f{M}^\r{cu}(A,B)\ot\f{M}^\r{cu}(A',B)\hspace{10mm}(A,A'\in\b{A}_\r{u},B\in\b{A}_\r{cu})$$}
\end{theorem}
We denote by $\f{op}:\b{A}_\r{nc}\to\b{A}_\r{nc}$ the functor that associates to any algebra the algebra with opposite multiplication.
\begin{theorem}\label{2009072327}
\emph{For any two algebras $A,B$ we have the natural isomorphism
$$\f{op}(\f{M}^\r{nc}({A,B}))\cong\f{M}^\r{nc}({\f{op}(A),\f{op}(B)}).$$}
\end{theorem}\begin{proof}
It is easily seen from the construction of $\f{M}^\r{nc}_{A,B}$.\end{proof}
\begin{proposition}
\emph{For $A,B,C\in\b{A}_\r{nc}/\b{A}_\r{u}$ and $*\in\{\r{nc,c,cr}\}/\{\r{u,cu,cur}\}$ there is a natural pro-morphism
\begin{equation*}%\label{2009051300}
\Phi_{A,B,C}:\f{M}^*_{A,C}\to\f{M}^*_{B,C}\ot\f{M}^*_{A,B}
\end{equation*}
such that\begin{equation}\label{2009061331}
(\r{id}_{\f{M}^*_{C,D}}\ot\Phi_{A,B,C})\Phi_{A,C,D}=(\Phi_{B,C,D}\ot\r{id}_{\f{M}^*_{A,B}})\Phi_{A,B,D}\end{equation}}
\end{proposition}\begin{proof}
By universality of $\f{M}_{A,C}$, there is a unique pro-morphism $\Phi_{A,B,C}$ satisfying
$$(\r{id}_C\ot\Phi_{A,B,C})\Upsilon_{A,C}=(\Upsilon_{B,C}\ot\r{id}_{\f{M}_{A,B}})\Upsilon_{A,B}.$$
Identity (\ref{2009061331}) follows from the universality of $\f{M}_{A,D}$.
\end{proof}
%%%%%%%%%%%%%%%%%%%%%%%%%%%%%%%%%%%%%%%%%%%%%%%%%%%%%%%%%%%%%%%%%%%%%%%%%%%%%%%%%%%%%%%%%%%%%%%%%%%%%%%%%%%%%%%%%%%%%%%%%%
%%%%%%%%%%%%%%%%%%%%%%%%%%%%%%%%%%%%%%%%%%%%%%%%%%%%%%%%%%%%%%%%%%%%%%%%%%%%%%%%%%%%%%%%%%%%%%%%%%%%%%%%%%%%%%%%%%%%%%%%%%
%%%%%%%%%%%%%%%%%%%%%%%%%%%%%%%%%%%%%%%%%%%%%%%%%%%%%%%%%%%%%%%%%%%%%%%%%%%%%%%%%%%%%%%%%%%%%%%%%%%%%%%%%%%%%%%%%%%%%%%%%%
%%%%%%%%%%%%%%%%%%%%%%%%%%%%%%%%%%%%%%%%%%%%%%%%%%%%%%%%%%%%%%%%%%%%%%%%%%%%%%%%%%%%%%%%%%%%%%%%%%%%%%%%%%%%%%%%%%%%%%%%%%
\section{Classical Algebraic Homotopy}\label{1910191623}
A compatible relation on a category $\b{C}$ is an equivalence relation $\f{R}$ on the class of morphisms of $\b{C}$ such that for morphisms
$f,f':C\to C'$ and $g,g':C'\to C''$ if $f\f{R}f'$ and $g\f{R}g'$ then $gf\f{R}g'f'$. Denote by $\r{Hot}_\f{R}(\b{C})$ the category whose objects are
those of $\b{C}$ and whose hom-sets are $$[C,C']_\f{R}=\r{Hom}_{\r{Hot}_\f{R}(\b{C})}(C,C'):=\r{Hom}_\b{C}(C,C')/\f{R}.$$

Let $\b{A}$ be a subcategory of $\b{A}_\r{nc}$. We say that $\b{A}$ is admissible if $\b{A}$ is closed under polynomial extensions, and if for every
algebra $A$ in $\b{A}$ the canonical embedding $\r{e}:A\to A[x]$ given by $a\mapsto a$, the evaluation morphisms $\r{p}_0,\r{p}_1:A[x]\to A$
given respectively by $x\mapsto0,x\mapsto1$, and the morphism defined by $x\mapsto 1-x$ from $A[x]$ onto $A[x]$, belong to $\b{A}$.
It is clear that $\b{A}_*$ for $*\in\{\r{nc,c,cr,u,cu,cur}\}$ is admissible.

Let $A$ be an admissible category of algebras.
Two morphisms $f,g:A\to B$ in $\r{pro}\text{-}\b{A}$ are said to be elementary homotopic, denoted $f\approx g$, if there is a morphism
$H:A\to B\ot\bb{F}[x]\cong B[x]$ in $\r{pro}\text{-}\b{A}$, called elementary homotopy from $f$ to $g$, such that $\r{p}_0H=f$ and
$\r{p}_1H=g$. Similarly, elementary homotopic morphisms in $\r{ind}\text{-}\b{A}$ are defined.

Two morphisms $f,g:A\to B$ in $\b{A}$ are called algebraic homotopic, denoted $f\r{h}g$, if there is a finite chain $h_0,\ldots,h_n$ of
morphisms in $\b{A}$ such that \begin{equation}\label{2011282013}f=h_0\approx h_1\approx\ldots\approx h_n=g.\end{equation}
It is easily verified that $\r{h}$ is a compatible relation on $\b{A}$, and accordingly we have the homotopy category
$\r{Hot}(\b{A})=\r{Hot}_\r{h}(\b{A})$.

We say that two pro-morphisms $f,g:A\to B$ are strongly homotopic, denoted $f\r{sh}g$, if there is a finite chain $h_0,\ldots,h_n$ of
pro-morphisms such that (\ref{2011282013}) is satisfied. It is easily verified that $\r{sh}$ is a compatible relation on
$\r{pro}\text{-}\b{A}$, and accordingly we have the homotopy category $\r{Hot}_\r{sh}(\r{pro}\text{-}\b{A})$.
The compatible relation $\r{sh}$ between ind-morphisms is defined similarly, and hence we have the category
$\r{Hot}_\r{sh}(\r{ind}\text{-}\b{A})$.

We say that two pro-morphisms $f,g:A\to B$ are weakly homotopic, denoted $f\r{wh}g$, if their images in
$\r{pro}\text{-}\r{Hot}(\b{A})$ are equal. It is clear that $\r{wh}$ is a compatible relation on $\r{pro}\text{-}\b{A}$,
and accordingly we have the homotopy category $\r{Hot}_\r{wh}(\r{pro}\text{-}\b{A})$. The homotopy category
$\r{Hot}_\r{wh}(\r{ind}\text{-}\b{A})$ is defined similarly.

It is clear that $\r{wh}$ is coarser than $\r{sh}$. Thus we have the following two sequences of functors all induced by $\r{id}$
in the obvious way: \begin{equation*}%\label{2011282048}
\r{pro}\text{-}\b{A}\to\r{Hot}_\r{sh}(\r{pro}\text{-}\b{A})\to
\r{Hot}_\r{wh}(\r{pro}\text{-}\b{A})\to\r{pro}\text{-}\r{Hot}(\b{A})\end{equation*}
\begin{equation*}%\label{2011291211}
\r{ind}\text{-}\b{A}\to\r{Hot}_\r{sh}(\r{ind}\text{-}\b{A})\to
\r{Hot}_\r{wh}(\r{ind}\text{-}\b{A})\to\r{ind}\text{-}\r{Hot}(\b{A})\end{equation*}
In each of the above two rows the first and second functors are full and the third one is faithful. It is customary
(\cite{CortinasThom1,Garkusha1,Garkusha3,Gersten2}) to denote the hom-sets of the categories $\r{pro}\text{-}\r{Hot}(\b{A})$ and
$\r{ind}\text{-}\r{Hot}(\b{A})$ just by $[?,?]$.
It is easily verified that if $A\in\r{pro}\text{-}\b{A}$ and $B\in\b{A}$ (or if $A\in\b{A}$ and
$B\in\r{ind}\text{-}\b{A}$) then there are natural identifications
\begin{equation}\label{2011282112}[A,B]_\r{sh}\equiv[A,B]_\r{wh}\equiv[A,B]\end{equation}
Note that for pro-algebras $A=(A_i)$ and $B=(B_j)$, we have $$[A,B]=\lim_{\infty\leftarrow j}\lim_{i\to\infty}[A_i,B_j].$$
Similarly, for ind-algebras $A=(A_i)$ and $B=(B_j)$, we have $$[A,B]=\lim_{\infty\leftarrow i}\lim_{j\to\infty}[A_i,B_j].$$
We remark that in general for pro- or ind-algebras $A,B$ the set $[A,B]$ is very \emph{bigger} than $[A,B]_\r{wh}$.

The following simple lemma shows that for morphism $f,g$ in $\b{A}_*$ where $*\in\{\r{c,cr,u,cu,cur}\}$,
$f$ and $g$ are homotopic in $\b{A}_*$ iff $f$ and $g$ are homotopic in $\b{A}_\r{nc}$.
\begin{lemma}
\emph{Let $A,B$ be unital algebras, $f,g:A\to B$ be morphisms in $\b{A}_\r{nc}$, and $H$ be an elementary homotopy in $\b{A}_\r{nc}$ from $f$ to $g$.
If $f\in\b{A}_\r{u}$ then $H,g\in\b{A}_\r{u}$.}
\end{lemma}
\begin{lemma}\label{1909231259}
\emph{Let $A,B\in\b{A}_\r{nc}$ and $C\in\r{pro}\text{-}\b{A}_\r{nc}$, and let $f,g:A\to B\ot C$ be pro-morphisms.
Suppose that $\ov{f},\ov{g}:\f{M}^\r{nc}_{A,B}\to C$ denote the pro-morphisms associated with $f,g$.
Then, $f\r{sh}g$ iff $\ov{f}\r{sh}\ov{g}$, and $f\r{wh}g$ iff $\ov{f}\r{wh}\ov{g}$.}
\end{lemma}
\begin{proof}If $H:A\to B\ot C[x]$ is an elementary homotopy between $f,g$ then the pro-morphism $\ov{H}:\f{M}_{A,B}\to C[x]$ associated with $H$
is an elementary homotopy between $\ov{f},\ov{g}$. Conversely, if $G:\f{M}_{A,B}\to C[x]$ is an elementary homotopy between $\ov{f},\ov{g}$ then
$(\r{id}_B\ot G)\Upsilon_{A,B}$ is an elementary homotopy between $f,g$.\\
Let $C=(C_i)_i$ where $C_i\in\b{A}_\r{nc}$, and let $f=(f_i)_i$ and $g=(g_i)_i$. Suppose $f\r{wh}g$. Then, $f_i\r{h}g_i$. It follows from the first
part that $\ov{f}_i\r{sh}\ov{g}_i$, and hence $\ov{f}_i\r{wh}\ov{g}_i$. Thus, $\ov{f}\r{wh}\ov{g}$. The converse is similar.\end{proof}
We have the following lemma that can be easily proved.
\begin{lemma}\label{2012022213}
\emph{The analogue of the statement of Lemma \ref{1909231259} holds for $A,C\in\b{A}_\r{nc}$ and $B\in\r{ind}\text{-}\b{A}_\r{nc}$.}
\end{lemma}
\begin{theorem}\label{2012041948}
\emph{For $A,B\in\b{A}_\r{nc}$ and $C\in\r{pro}\text{-}\b{A}_\r{nc}$ ($A,C\in\b{A}_\r{nc}$ and $B\in\r{ind}\text{-}\b{A}_\r{nc}$), we have the
following canonical identification:\begin{equation*}%\label{2009201859}
[A,B\ot C]\cong[\f{M}^\r{nc}(A,B),C]\end{equation*}}
\end{theorem}
\begin{proof}It follows from Theorem \ref{2011302233}, Identities (\ref{2011282112}), and Lemmas \ref{1909231259} and \ref{2012022213}.\end{proof}
\begin{theorem}\label{1910191608}
\emph{The functor $\f{M}^\r{nc}$ preserves homotopy in the sense that for morphisms $f_0,f_1:A\to A'$ and $g_0,g_1:B\to B'$ in $\b{A}_\r{nc}$,
if $f_0\r{h}f_1$ and $g_0\r{h}g_1$ then $\f{M}_{f_0,g_0}\r{sh}\f{M}_{f_1,g_1}$. }
\end{theorem}
\begin{proof}It follows from definitions of $\f{M}_{f_0,g_0},\f{M}_{f_1,g_1}$ and Lemma \ref{1909231259}.\end{proof}
Note that the above four results hold for all $\f{M}^*$.

A functor $F$ from an admissible category $\b{A}$ of algebras to an arbitrary category $\b{C}$ is said to be homotopy invariant
if for any two morphism $f,g\in\b{A}$, $f\r{h}g$ implies $F(f)=F(g)$. We say that $F:\b{Aff}\to\b{C}$ is $\bb{A}^1$-homotopy invariant
if $F$ as the functor $\b{A}_\r{cu}\to\b{C}^\r{op}$, is homotopy invariant. The following lemma is very well-known.
\begin{lemma}\label{2010111932}
\emph{For any functor $F:\b{A}\to\b{C}$ the following statements are equivalent.\begin{enumerate}
\item[(i)] $F$ is homotopy invariant.
\item[(ii)] For every algebra $A$ in $\b{A}$, $F(\r{e}):F(A)\to F(A[x])$ is an isomorphism in $\b{C}$.
\item[(iii)] For every algebra $A$ in $\b{A}$, $F(\r{p}_0),F(\r{p}_1):F(A[x])\to F(A)$ are equal. \end{enumerate}}
\end{lemma}
\begin{proof}(i)$\Rightarrow$(ii): We show that $F(\r{p}_0)$ is the inverse of $F(\r{e})$ in $\b{C}$: We have $\r{p}_0\r{e}=\r{id}_A$. Thus
$F(\r{p}_0)F(\r{e})=\r{id}_{F(A)}$. The morphism $A[x]\to A[x]([y])$ given by $p(x)\mapsto p(xy)$ is an elementary homotopy between $\r{e}\r{p}_0$
and $\r{id}_{A[x]}$. Thus $F(\r{e})F(\r{p}_0)=\r{id}_{F(A[x])}$. (ii)$\Rightarrow$(iii): We have $\r{p}_0\r{e}=\r{p}_1\r{e}$. Thus
$F(\r{p}_0)F(\r{e})=F(\r{p}_1)F(\r{e})$. Since $F(\r{e})$ is an isomorphism, we must have $F(\r{p}_0)=F(\r{p}_1)$. (iii)$\Rightarrow$(i) is trivial.
\end{proof}
Recall that a differential graded commutative algebra (DGCA) is given by a pair $(\bar{A},d)$ where $\bar{A}=\oplus_{n=0}^\infty A^n$ is a GA
with a GC multiplication (i.e. $a_ia_j=(-1)^{ij}a_ja_i$ for $a_i\in A^i,a_j\in A^j$) and where $d:\bar{A}\to\bar{A}$ is a
GD  (i.e. a $\bb{F}$-linear map of degree $1$ such that $d(a_ia_j)=d(a_i)a_j+(-1)^ia_id(a_j)$ and $dd=0$). It is well-known that for every
$A\in\b{A}_\r{cu}$ there exists a DGCA $\Omega(A)=\oplus\Omega^n(A)$ with $\Omega^0(A)=A$ satisfying the following universal property:
For every unital DGCA $(\bar{B},\delta)$, any arbitrary morphism $f:A\to B^0$ in $\b{A}_\r{cu}$ extends uniquely to a morphism $f:\Omega(A)\to\bar{B}$
satisfying $fd=f\delta$. Indeed, $\Omega(A)$ is the exterior algebra associated to $A$-module $\Omega^1(A)$ and
$d:A\to\Omega^1(A)$ is a derivation such that for every unital $A$-module $M$ and any derivation $\delta:A\to M$ there is a unique
module homomorphism $\phi:\Omega^1(A)\to M$ satisfying $\delta=\phi d$. If $\bb{F}$ is real field and $A$ is the algebra of smooth real functions
on a smooth compact manifold then $\Omega(A)$ is isomorphic to de Rham complex of differential forms on the manifold
\cite[Proposition 8.1]{GraciaVarillyFigueroa1}. So, for any $A\in\b{A}_\r{cu}$ and its associated affine scheme $S=\c{S}(A)$
the cohomology groups of $\Omega(A)$ is denoted by $\c{H}_\r{deR}^n(A)$ or $\c{H}_\r{deR}^n(S)$ ($n\geq0$)
and called de Rham cohomology groups of $A$ or $S$. Note that $\c{H}_\r{deR}(S):=\oplus_{n=0}^\infty\c{H}_\r{deR}^n(S)$
is a unital GCA with the multiplication induced by that of $\Omega(A)$. By universality of $\Omega$ we may consider the functors
$$\c{H}_\r{deR}^0:\b{Aff}^\r{op}\to\b{A}_\r{cu},\hspace{2mm}\c{H}_\r{deR}^n:\b{Aff}^\r{op}\to\b{Vec},
\hspace{2mm}\c{H}_\r{deR}:\b{Aff}^\r{op}\to\b{A}_\r{u}.$$
It is shown that these functors are $\bb{A}^1$-homotopy invariant provided that $\r{Char}(\bb{F})=0$:
For $A\in\b{A}_\r{cu}$ let $\r{p}_0,\r{p}_1:\Omega(A[x])\to\Omega(A)$ denote the cochain maps induced by $\r{p}_0,\r{p}_1:A[x]\to A$.
By freeness of the variable $x$ we have a $A[x]$-module decomposition $\Omega^n(A[x])=\oplus_{i=0}^nM_i$ such that
$M_i$ is the $A[x]$-submodule generated by all elements $(d^ix)\omega_{n-i}\in\Omega^n(A[x])$ where
$\omega_{n-i}\in\Omega^{n-i}(A)$ and where $d^ix$ denotes the (exterior) product of $i$ copies of $dx$ in $\Omega(A[x])$.
Consider a linear map $\phi^{n}$ from $\Omega^n(A[x])$ into $\Omega^{n-1}(A)$ defined for $\alpha_i\in M_i$ by $\phi^n(\alpha_i):=0$ if $i\neq1$
and $\phi^n(\alpha_1):=\int_0^1\alpha_1$. Here the formal integral is given by
$$\int_0^1dx(x^k\omega):=\frac{1}{k+1}x^{k+1}\omega\mid_0^1=\frac{1}{k+1}\omega\hspace{3mm}(\omega\in\Omega^{n-1}(A))$$
Then it can be checked that $(\phi^n)_n$ is a cochain homotopy between $\r{p}_0$ and $\r{p}_1$,
and hence $\c{H}_\r{deR}^n(\r{p}_0)=\c{H}_\r{deR}^n(\r{p}_1)$.
%%%%%%%%%%%%%%%%%%%%%%%%%%%%%%%%%%%%%%%%%%%%%%%%%%%%%%%%%%%%%%%%%%%%%%%%%%%%%%%%%%%%%%%%%%%%%%%%%%%%%%%%%%%%%%%%%%%%%%%%%
%%%%%%%%%%%%%%%%%%%%%%%%%%%%%%%%%%%%%%%%%%%%%%%%%%%%%%%%%%%%%%%%%%%%%%%%%%%%%%%%%%%%%%%%%%%%%%%%%%%%%%%%%%%%%%%%%%%%%%%%%%
%%%%%%%%%%%%%%%%%%%%%%%%%%%%%%%%%%%%%%%%%%%%%%%%%%%%%%%%%%%%%%%%%%%%%%%%%%%%%%%%%%%%%%%%%%%%%%%%%%%%%%%%%%%%%%%%%%%%%%%%%%
%%%%%%%%%%%%%%%%%%%%%%%%%%%%%%%%%%%%%%%%%%%%%%%%%%%%%%%%%%%%%%%%%%%%%%%%%%%%%%%%%%%%%%%%%%%%%%%%%%%%%%%%%%%%%%%%%%%%%%%%%%
\section{Homotopy Invariant Subalgebras}\label{1909241128}
The main aim of this section is to consider an analogue of the functor $\pi_0$ for algebras:
\begin{definition}\label{2010231857}
\emph{Let $A\in\b{A}_\r{nc}/\b{A}_\r{u}$. For $*\in\{\r{nc,c,cr}\}/\{\r{u,cu,cur}\}$, let
$$\f{p}_0,\f{p}_1:A\to\underleftarrow{\lim}\f{M}^*_{A,\bb{F}[x]}$$ be given by
$$\f{p}_0:=\underleftarrow{\lim}\big[((x\mapsto0)\ot\r{id})\Upsilon^*_{A,\bb{F}[x]}\big]\hspace{2mm}\text{and}\hspace{2mm}
\f{p}_1:=\underleftarrow{\lim}\big[((x\mapsto1)\ot\r{id})\Upsilon^*_{A,\bb{F}[x]}\big].$$
We let the subalgebra $\f{P}^*(A)\subseteq A$ be defined by
$$\f{P}^*(A):=\{a\in A:\f{p}_0(a)=\f{p}_1(a)\}.$$}\end{definition}
We have the following useful lemma.
\begin{lemma}\label{2009301424}
\emph{With assumptions of Definition \ref{2010231857}, suppose $A$ is generated by $G$
and $V$ is a vector basis for $\bb{F}[x]$. For any $a\in A$ the following statements are equivalent.
\begin{enumerate}
\item[(i)] For every $\theta\in\Theta_{G,V}$, there exists $\hat{a}_\theta\in\f{M}^*\theta$ such that $\Upsilon^*\theta(a)=1\ot\hat{a}_\theta$.
\item[(ii)] For every algebra $C\in\b{A}_*$ and every morphism $\phi:A\to C[x]$ in $\b{A}_\r{nc}/\b{A}_\r{u}$, $\phi(a)$ is constant.
\item[(iii)] For every $\theta\in\Theta_{G,V}$, $((x\mapsto0)\ot\r{id})\Upsilon^*\theta(a)=((x\mapsto1)\ot\r{id})\Upsilon^*\theta(a)$.
\item[(iv)] For every $C,\phi$ as in (ii), $(x\mapsto0)\phi(a)=(x\mapsto1)\phi(a)$.
\item[(v)] $a$ belongs to $\f{P}^*(A)$.
\end{enumerate}}\end{lemma}
\begin{proof}
(i)$\Rightarrow$(ii) and (iii)$\Rightarrow$(iv) follow from the universal property of $\f{M}_{A,\bb{F}[x]}$. (i)$\Rightarrow$(iii) and (iv)$\Rightarrow$(iii)
are trivial. Also, (ii)$\Rightarrow$(i),(iv) and (iii)$\Leftrightarrow$(v) are trivial.
If $\phi:A\to C[x]$ is a morphism then we can also consider the morphism $A\to (C[x])[y]$ given by
$b\mapsto\phi(b)(xy)$. This shows that (iv)$\Rightarrow$(ii) is satisfied.\end{proof}
It is easily verified that any morphism from $A$ to $B$ transforms $\f{P}^*(A)$ into $\f{P}^*(B)$. Thus we may consider the following subfunctor
of $\r{id}_{\b{A}_\r{nc}/\b{A}_\r{u}}$: $$\f{P}^*:\b{A}_\r{nc}/\b{A}_\r{u}\to\b{A}_\r{nc}/\b{A}_\r{u}.$$
(We say that a functor $F:\b{A}_*\to\b{A}_*$ is a subfunctor of $\r{id}_{\b{A}_*}$ if for any $A\in\b{A}_*$, $F(A)$ is a subalgebra
of $A$ (in unital cases, $F(A)$ is required to have the unit of $A$) and for any morphism $f:A\to B$, $F(f)=f|_{F(A)}$.)
For any $A\in\b{A}_\r{nc}/\b{A}_\r{u}$, we have
$$\f{P}^\r{nc}(A)\subseteq\f{P}^\r{c}(A)\subseteq\f{P}^\r{cr}(A),\hspace{10mm}\f{P}^\r{u}(A)\subseteq\f{P}^\r{cu}(A)\subseteq\f{P}^\r{cur}(A),$$
and for any unital algebra $A$, $\f{P}^\r{nc}(A)\subseteq\f{P}^\r{u}(A)$, $\f{P}^\r{c}(A)\subseteq\f{P}^\r{cu}(A)$,
$\f{P}^\r{cr}(A)\subseteq\f{P}^\r{cur}(A)$.
\begin{theorem}\label{2010172101}
\emph{$\f{P}^*:\b{A}_*\to\b{A}_*$ for $*\in\{\r{nc,c,cr,u,cu,cur}\}$ is homotopy invariant.}\end{theorem}
\begin{proof}
For morphisms $f,g:A\to B$ in $\b{A}_*$ suppose $H:A\to B[x]$ is an elementary homotopy in $\b{A}_*$ from $f$ to $g$.
For any $a\in\f{P}^*(A)$, $H(a)$ is constant and hence $f(a)=\r{p}_0H(a)=\r{p}_1H(a)=g(a)$.
Thus $\f{P}^*(f)=\f{P}^*(g)$. The proof is complete.\end{proof}
\begin{theorem}\label{2012091918}
\emph{For any algebra $A$, we have $\f{P}^\r{nc}(A)=0$.}\end{theorem}
\begin{proof}
The morphism $A\to\c{M}_2(A)[x]$ defined by
$$a\mapsto\left(\begin{array}{cc}a & 0 \\0 & 0 \\\end{array}\right)+\left(\begin{array}{cc}0 & a \\0 & 0 \\\end{array}\right)x$$
transforms any nonzero $a$ to a nonconstant polynomial. The proof is complete.\end{proof}
It follows easily from Theorems \ref{2012091919} and \ref{2012091918} that for any algebra $A$, $$\f{P}^\r{u}(A^+)=\bb{F}1.$$
\begin{theorem}\label{2010312353}
\emph{For any $A\in\b{A}_\r{u}$ we have $\f{P}^\r{c}(A)=\f{P}^\r{cu}(A)$ and $\f{P}^\r{cr}(A)=\f{P}^\r{cur}(A)$.}\end{theorem}
\begin{proof}
We have $\f{P}^\r{c}(A)\subseteq\f{P}^\r{cu}(A)$.
Let $a\in\f{P}^\r{cu}(A)$. For $C\in\b{A}_\r{c}$ and $\phi:A\to C[x]$ it is easily verified that $\phi(1)\in C$
and $\phi(A)\subseteq\hat{C}[x]$ where $\hat{C}:=\{c\in C:\phi(1)c=c\}$.
Thus $\phi:A\to\hat{C}[x]$ is a unital morphism and hence
$\phi(a)\in\hat{C}\subseteq C$. So $\f{P}^\r{cu}(A)\subseteq\f{P}^\r{c}(A)$. The other case is similar.\end{proof}
\begin{theorem}
\emph{For any algebra $A$ we have $\f{P}^\r{c}(A)=\r{p}^{-1}\f{P}^\r{c}(A_\r{com})$ where $\r{p}:A\to A_\r{com}$ denotes the canonical projection.}
\end{theorem}
\begin{proof}
It follows from Lemma \ref{2009301424}(ii) and the fact that any morphism $\phi:A\to C[x]$ where $C$ is a commutative algebra factors through $\r{p}$.
\end{proof}
In general, it is not easy to determine the subalgebras $\f{P}^*(A)$.
We shall consider some other homotopy invariant subfunctors of $\r{id}_{\b{A}_*}$ which are computable and may be applied for our proposes
instead of $\f{P}^*$. In $\S$ \ref{2010141800} we will show that $\f{P}^*$ is the largest subfunctor among all
homotopy invariant subfunctors of $\r{id}_{\b{A}_*}$.
\begin{lemma}\label{2010180840}
\emph{Let $C$ be an algebra and $p=\sum_{i=0}^nc_ix^i$ be a polynomial in $C[x]$.
\begin{enumerate}
\item[(i)] If $C$ is commutative and if $p^k=p$ for some $k\geq2$ with the property $\r{Char}(\bb{F})\nmid(k-1)$, then $p$ is constant.
\item[(ii)] If $C$ is unital and has no nonzero nilpotent element and if $p$ is a root of a nonzero polynomial with coefficients
in $\bb{F}$, then $p$ is constant.
\end{enumerate}}\end{lemma}
\begin{proof}
(i) We have $c_0^k=c_0$ and $kc_0^{k-1}c_1=c_1$. Thus $kc_0c_1=c_0c_1$. This implies that $c_0c_1=0$ and therefore $c_1=0$.
Analogously and respectively, it is proved that $c_2=0,\ldots,c_n=0$. (ii) There are $\lambda_0,\ldots,\lambda_{k-1}\in\bb{F}$ with $k\geq1$
such that $p^k+\lambda_{k-1}p^{k-1}+\cdots+\lambda_0=0$. Thus $c_n^k=0$ and then, by the assumption, we have $c_n=0$. The proof is complete.
\end{proof}
\begin{proposition}\label{1909151849}
\emph{Let $A$ be an algebra and let $a\in A$.
\begin{enumerate}
\item[(i)] If $a^k=a$ for some $k\geq2$ with $\r{Char}(\bb{F})\nmid(k-1)$ then $a\in\f{P}^\r{c}(A)$.
\item[(ii)] If $a$ is an idempotent then $a\in\f{P}^\r{c}(A)$.
\item[(iii)] If $A$ is unital and $a$ is integral over $\bb{F}$ then $a\in\f{P}^\r{cur}(A)$.
\end{enumerate}}
\end{proposition}
\begin{proof}
(i) and (iii) follow from Lemma \ref{2010180840} and the definition of $\f{P}^\r{c}$ by the analogue of Lemma \ref{2009301424}(ii).
(ii) follows from (i).
\end{proof}
For any algebra $A$ and every $k\geq2$ let $\f{I}_k\f{P}(A)\subseteq A$ denote the subalgebra generated by all elements $a$ with the property
$a^k=a$. Let also $\f{IP}(A)\subseteq A$ denote the subalgebra generated by all elements $a$ such that $a^k=a$ for some $k\geq2$.
For any commutative unital algebra $A$ let $\r{Int}(A)\subseteq A$ denote integral closure of $\bb{F}$ in $A$.
It is clear that $\f{I}_k\f{P}(A)\subseteq\f{IP}(A)$ and $\f{IP}(A)\subseteq\r{Int}(A)$.
We have also the subfunctors
$$\f{IP},\f{I}_k\f{P}:\b{A}_\r{nc}\to\b{A}_\r{nc}\hspace{3mm}\text{and}\hspace{3mm}\r{Int}:\b{A}_\r{cu}\to\b{A}_\r{cu}.$$
\begin{theorem}\label{2010172015}
\emph{\begin{enumerate}
\item[(i)] For any algebra $A$, if $\r{Char}(\bb{F})\nmid(k-1)$ then $\f{I}_k\f{P}(A)\subseteq\f{P}^\r{c}(A)$, and
if $\r{Char}(\bb{F})=0$ then $\f{IP}(A)\subseteq\f{P}^\r{c}(A)$.
\item[(ii)] If $A$ is commutative and unital then $\r{Int}(A)\subseteq\f{P}^\r{cur}(A)=\f{P}^\r{cr}(A)$.
\item[(iii)] If $\r{Char}(\bb{F})\nmid(k-1)$ then $\f{I}_k\f{P}:\b{A}_\r{c}\to\b{A}_\r{c}$ is homotopy invariant.
If $\r{Char}(\bb{F})=0$ then $\f{IP}:\b{A}_\r{c}\to\b{A}_\r{c}$ is homotopy invariant.
\item[(iv)] The functor $\r{Int}:\b{A}_\r{cur}\to\b{A}_\r{cur}$ is homotopy invariant.
\end{enumerate}}
\end{theorem}
\begin{proof}
(i) and (ii) follow from Proposition \ref{1909151849}. Proofs of (iii) and (iv) are similar to the proof of Theorem \ref{2010172101} and
follow from Lemma \ref{2010180840}.
\end{proof}
It is easily verified that any subfunctor of $\r{id}_{\b{A}_*}$ for $*\in\{\r{nc,c,cr}\}$ preserves arbitrary direct sums. Thus we have,
\begin{theorem}\label{2010040928}
\emph{$\f{I}_k\f{P},\f{IP},\f{P}^\r{cr},\f{P}^\r{c}:\b{A}_\r{nc}\to\b{A}_\r{nc}$ preserve arbitrary direct sums.}
\end{theorem}
For proof of the next lemma we shall need the following fact from elementary Linear Algebra:
Let $W,W'$ be vector spaces and $W_0\subseteq W$ be a vector subspace. Let $x\in W'\ot W$ be such that for every linear functional
$\alpha:W\to\bb{F}$ with $\alpha|_{W_0}=0$ we have $(\r{id}_{W'}\ot\alpha)(x)=0$. Then $x$ belongs to $W'\ot W_0$.
\begin{lemma}\label{2010251146}
\emph{Suppose $\bb{F}$ is algebraically closed. Let $A\in\b{A}_\r{curfg}$, $W$ be a vector space, and $W_0\subseteq W$ be a vector subspace.
For any $c\in A\ot W$ if $h\ot\r{id}_W(c)\in W_0$ for every $h\in\r{Pnt}(A)$, then $c\in A\ot W_0$.}
\end{lemma}
\begin{proof}
Let $Z\subseteq\bb{A}^m$ be the zero-locus of a radical ideal $I\subseteq\bb{F}[X_m]$ with $A\cong\bb{F}[X_m]/I$.
We know $Z$ may be identified with $\r{Pnt}_{\neq0}(A)$ through $\zeta\mapsto(p(X_m)\mapsto p(\zeta))$, and
by Hilbert's Nullstellensatz, $A\ot W$ may be identified with vector space of all functions $f:Z\to W$ such that
$f(\zeta)=\sum_i\zeta^iw_i$ for some polynomial $\sum_iw_iX_m^i$ with $w_i\in W$. So, suppose that $f$ be such a representation for $c$.
By assumptions we have $f(\zeta)\in W_0$ for every $\zeta\in Z$. Thus for any linear functional $\alpha:W\to\bb{F}$ with $\alpha|_{W_0}=0$
we have $\r{id}_A\ot\alpha(c)=\alpha\circ f=0$. Hence $c\in A\ot W_0$.
\end{proof}
We have the following corollary of Lemma \ref{2010251146}.
\begin{proposition}\label{2010251242}
\emph{Suppose $\bb{F}$ is algebraically closed and let $F:\b{A}_\r{cur}\to\b{A}_\r{cur}$ be any subfunctor of $\r{id}_{\b{A}_\r{cur}}$. Then
for every $A,B\in\b{A}_\r{curfg}$ we have $F(A\ot B)=F(A)\ot F(B)$.}
\end{proposition}
\begin{proof}
Let $c\in F(A\ot B)\subseteq A\ot B$. For any $h\in\r{Pnt}(A)$, $h\ot\r{id}_B:A\ot B\to B$ is a morphism in $\b{A}_\r{cur}$ and hence
$h\ot\r{id}_B(c)\in F(B)\subseteq B$. Thus it follows from Lemma \ref{2010251146} that $c\in A\ot F(B)$. Similarly, $c\in F(A)\ot B$.
Thus $c\in F(A)\ot F(B)$.\\
Since $x\mapsto x\ot1_B$ and $y\mapsto1_A\ot y$ define two morphisms $A\to A\ot B$ and $B\to A\ot B$ in $\b{A}_\r{cur}$, we have
$a\ot1_B\in F(A\ot B)$ and $1_A\ot b\in F(A\ot B)$ for every $a\in F(A)$ and $b\in F(B)$. Thus $F(A)\ot F(B)\subseteq F(A\ot B)$.
\end{proof}
The following theorem follows directly from Proposition \ref{2010251242}.
\begin{theorem}\label{1909210027}
\emph{Suppose $\bb{F}$ is algebraically closed. Then the functors
$$\f{P}^\r{c},\f{P}^\r{cr},\f{IP},\f{I}_k\f{P},\r{Int}:\b{A}_\r{curfg}\to\b{A}_\r{cur}$$ preserve tensor product.}
\end{theorem}
For any affine scheme $S=\c{S}(A)$ ($A\in\b{A}_\r{cu}$) the affine scheme of path-connected components of $S$
is defined to be the affine scheme \begin{equation}\label{2012082236}\pi_0(S):=\c{S}(\f{P}^\r{cu}(A)).\end{equation}
Thus $\pi_0$ may be considered as a $\bb{A}^1$-homotopy invariant functor $$\pi_0:\b{Aff}\to\b{Aff}.$$
Recall that an affine group-scheme is a group object $G$ in $\b{Aff}$. Suppose that $G=\c{S}(A)$. Then the group structure of $G$ is induced
by a Hopf algebra structure on $A$. In case $\bb{F}$ is algebraically closed and $G$ is an affine variety, $G$ is called affine group-variety.
Note that in this case $A$ is a finitely generated integral domain.
\begin{theorem}
\emph{For any affine group-variety $G$, $\pi_0(G)$ is an affine group-scheme.}
\end{theorem}
\begin{proof}
It follows from Theorem \ref{1909210027}.
\end{proof}
For a discussion on path-connected components of algebraic groups see \cite[Chapter 2]{Jardine2}.
The following important theorem and its proof have been kindly offered to me by Professor Corti\~{n}as.
\begin{theorem}\label{2010311847}
\emph{Suppose that $\r{Char}(\bb{F})=0$. Then $$\pi_0(S)=\c{H}_\r{deR}^0(S)\hspace{10mm}(S\in\b{Aff}).$$}
\end{theorem}
\begin{proof}
Let $S=\c{S}(A)$. Consider the universal derivation $d:A\to\Omega^1(A)$ as described in $\S$\ref{1910191623}. We must show that
$\ker(d)=\f{P}^\r{cu}(A)$. Let $\bar{A}$ denote the commutative unital algebra with underlying vector space $A\oplus\Omega^1(A)$ and multiplication
given by $(a,\omega)(a',\omega')=(aa',a\omega'+a'\omega)$ for $a,a'\in A$ and $\omega,\omega'\in\Omega^1(A)$. Then the assignment
$a\mapsto a+d(a)x$ defines a morphism from $A$ into $\bar{A}[x]$. Thus if $a\in\f{P}^\r{cu}(A)$ then $d(a)=0$.
Hence $\f{P}^\r{cu}(A)\subseteq\ker(d)$.
Suppose that $a\in\ker(d)$. Let $\phi:A\to B[x]$ be an arbitrary unit preserving morphism for a commutative unital algebra $B$. $B[x]$ can be
considered as a unital $A$-module through $\phi$. Let $D:B[x]\to B[x]$ denote the ordinary derivation given by
$D(\sum_{i=0}^nb_ix^i)=\sum_{i=1}^nib_ix^{i-1}$. Then $D\phi:A\to B[x]$ is a derivation into the $A$-module $B[x]$. Thus $D\phi(a)=0$.
This means that $\phi(a)$ must be a constant polynomial and hence $a\in\f{P}^\r{cu}(A)$. The proof is complete.
\end{proof}
%%%%%%%%%%%%%%%%%%%%%%%%%%%%%%%%%%%%%%%%%%%%%%%%%%%%%%%%%%%%%%%%%%%%%%%%%%%%%%%%%%%%%%%%%%%%%%%%%%%%%%%%%%%%%%%%%%%%%%%%%%
%%%%%%%%%%%%%%%%%%%%%%%%%%%%%%%%%%%%%%%%%%%%%%%%%%%%%%%%%%%%%%%%%%%%%%%%%%%%%%%%%%%%%%%%%%%%%%%%%%%%%%%%%%%%%%%%%%%%%%%%%%
%%%%%%%%%%%%%%%%%%%%%%%%%%%%%%%%%%%%%%%%%%%%%%%%%%%%%%%%%%%%%%%%%%%%%%%%%%%%%%%%%%%%%%%%%%%%%%%%%%%%%%%%%%%%%%%%%%%%%%%%%%
%%%%%%%%%%%%%%%%%%%%%%%%%%%%%%%%%%%%%%%%%%%%%%%%%%%%%%%%%%%%%%%%%%%%%%%%%%%%%%%%%%%%%%%%%%%%%%%%%%%%%%%%%%%%%%%%%%%%%%%%%%
\section{Costrict homotopization}\label{2010141800}
\begin{definition}\label{2010111257}
\emph{Let $\b{A}$ be an admissible category of algebras and $\b{C}$ be an arbitrary category.
The costrict homotopization of a functor $F:\b{A}\to\b{C}$ is defined to be a homotopy invariant functor
$\lfloor F\rfloor:\b{A}\to\b{C}$ coming with a natural transformation $\alpha_F:F\to\lfloor F\rfloor$ such that
the following universal property is satisfied: For any homotopy invariant functor $G:\b{A}\to\b{C}$ and any natural transformation
$\beta:F\to G$ there is a unique natural transformation $\underline{\beta}:\lfloor F\rfloor\to G$ such that $\beta=\underline{\beta}\alpha_F$.}
\end{definition}
It is clear that if $\lfloor F\rfloor$ exists then it is unique up to a natural isomorphism. Also, for any homotopy invariant functor $F$
we have $\lfloor F\rfloor=F$.

Note that Definition \ref{2010111257} is a simple generalization of Weibel's concept of strict
homotopization (\cite{Weibel1,Garkusha1}). It seems that this concept at the first time has been considered by Gersten \cite{Gersten2}.
\begin{theorem}\label{2012081413}
\emph{Suppose $\b{C}$ has finite coequalizers. Then any functor $F:\b{A}\to\b{C}$ has a costrict homotopization.}
\end{theorem}
\begin{proof}
For any $A\in\b{A}$ let $\alpha_F(A):F(A)\to\lfloor F\rfloor(A)$ be a coequalizer of
$$F(\r{p}_0),F(\r{p}_1):F(A[x])\to F(A)$$
in $\b{C}$. For any morphism $f:A\to B$ in $\b{A}$, the coequalizer property of $\lfloor F\rfloor(A)$ implies the existence of a unique morphism
$\lfloor F\rfloor(f)$ in $\b{C}$ making the diagram
$$\xymatrix{F(A[x])\ar[d]_-{F(f[x])}\ar@<2.5pt>[r]^{F(\r{p}_0)}\ar@<-2.5pt>[r]_{F(\r{p}_1)}&F(A)\ar[d]_{F(f)}
\ar[r]^-{\alpha_F(A)}&\lfloor F\rfloor(A)\ar[d]_{\lfloor F\rfloor(f)}\\
F(B[x])\ar@<2.5pt>[r]^{F(\r{p}_0)}\ar@<-2.5pt>[r]_{F(\r{p}_1)}&F(B)\ar[r]^-{\alpha_F(B)}&\lfloor F\rfloor(B)}$$
commutative. Thus $\lfloor F\rfloor$ may be considered as a functor. Suppose that $H:A\to B[x]$ is an elementary homotopy from $f$ to $g$. We have
\begin{equation*}\begin{split}
\alpha_F(B)F(f)&=\alpha_F(B)(F(\r{p}_0)F(H))=(\alpha_F(B)F(\r{p}_0))F(H)\\
&=(\alpha_F(B)F(\r{p}_1))F(H)=\alpha_F(B)(F(\r{p}_1)F(H))=\alpha_F(B)F(g)\end{split}\end{equation*}
Thus $\lfloor F\rfloor(f)=\lfloor F\rfloor(g)$, and therefore $\lfloor F\rfloor$ is homotopy invariant.
Suppose $G:\b{A}\to\b{C}$ is homotopy invariant and $\beta:F\to G$ is a natural transformation.
By Lemma \ref{2010111932}(iii), $G(\r{p}_0)=G(\r{p}_1)$, and hence universality of the coequalizer $\lfloor F\rfloor(A)$
implies the existence of a morphism $\underline{\beta}(A)$ making the following diagram commutative:
$$\xymatrix{F(A[x])\ar[d]_-{\beta(A[x])}\ar@<2.5pt>[r]^{F(\r{p}_0)}\ar@<-2.5pt>[r]_{F(\r{p}_1)}&F(A)\ar[d]_{\beta(A)}
\ar[r]^-{\alpha_F(A)}&\lfloor F\rfloor(A)\ar[dl]^{\underline{\beta}(A)}\\
G(A[x])\ar@<2.5pt>[r]^{G(\r{p}_0)}\ar@<-2.5pt>[r]_{G(\r{p}_1)}&G(A)}$$\end{proof}
It can be seen easily from the proof of Theorem \ref{2012081413} that for the identity functor $\r{id}:\b{A}\to\b{A}$ we have
$\lfloor\r{id}\rfloor=0$. (Indeed, it follows from the fact that for any $A\in\b{A}$ and every $a,b\in A$, $\r{p}_0(a+(b-a)x)=a$ and
$\r{p}_1(a+(b-a)x)=b$.) Also, for any affine scheme $S=\c{S}(A)$ regarded as the functor
\begin{equation}\label{2012082241}S=\r{Hom}_{\b{A}_\r{cu}}(A,?):\b{A}_\r{cu}\to\b{Set}\end{equation}
we have \begin{equation}\label{2012082231}\lfloor S\rfloor=[A,?].\end{equation}

A cotype for the concept of strict homotopization may be defined as follows.
\begin{definition}
\emph{Let $\b{A}$ be an admissible category of algebras and $\b{C}$ be an arbitrary category.
The costrict homotopization of a functor $F:\b{A}\to\b{C}$ is defined to be a homotopy invariant functor
$\lceil F\rceil:\b{A}\to\b{C}$ coming with a natural transformation $\alpha^F:\lceil F\rceil\to F$ such that
the following universal property is satisfied: For any homotopy invariant functor $G:\b{A}\to\b{C}$ and any natural transformation
$\beta:G\to F$ there is a unique natural transformation $\ov{\beta}:G\to\lceil F\rceil$ such that $\beta=\alpha^F\ov{\beta}$.}
\end{definition}
It is clear that if $\lceil F\rceil$ exists then it is unique up to a natural isomorphism. Also, for any homotopy invariant functor $F$
we have $\lceil F\rceil=F$.

The following theorem is one of the main results of this note.
\begin{theorem}\label{2010141245}
\emph{Suppose $\b{C}$ has finite equalizers and arbitrary inverse limits. Then any functor $F:\b{A}_*\to\b{C}$, where
$*\in\{\r{nc,c,cr,u,cu,cur}\}$, has a costrict homotopization.}
\end{theorem}
\begin{proof}
Let $\b{A},\f{M},\Upsilon$ denote $\b{A}_*,\f{M}^*,\Upsilon^*$. For any $A\in\b{A}$, let $\alpha^F(A):\lceil F\rceil(A)\to F(A)$ denote the
equalizer of the morphisms $$\hat{\f{p}}_0,\hat{\f{p}}_1:F(A)\to\underleftarrow{\lim}F(\f{M}_{A,\bb{F}[x]}),$$
given (similarly to $\f{p}_0,\f{p}_1$ in Definition \ref{2010231857}) by
$$\hat{\f{p}}_0:=\underleftarrow{\lim}F[((x\mapsto0)\ot\r{id})\Upsilon_{A,\bb{F}[x]}]\hspace{2mm}\text{and}\hspace{2mm}
\hat{\f{p}}_1:=\underleftarrow{\lim}F[((x\mapsto1)\ot\r{id})\Upsilon_{A,\bb{F}[x]}].$$
For any $f:A\to B$ in $\b{A}$, by the universal property of the equalizer $\lceil F\rceil(B)$, there is a unique morphism
$\lceil F\rceil(f)$ making the following diagram commutative:
$$\begin{array}{c}
    \xymatrix{\lceil F\rceil(A)\ar[d]_{\lceil F\rceil(f)}\ar[r]^-{\alpha^F(A)}&\\
\lceil F\rceil(B)\ar[r]_-{\alpha^F(B)}&}
  \end{array}
\underleftarrow{\lim}F\left(
    \begin{array}{c}
     {\xymatrix{A\ar[r]^-{\Upsilon_{A,\bb{F}[x]}}\ar[d]_-{f}&\bb{F}[x]\ot\f{M}_{A,\bb{F}[x]}
\ar@<2.5pt>[r]^-{x\mapsto0}\ar@<-2.5pt>[r]_-{x\mapsto1}&
\f{M}_{A,\bb{F}[x]}\ar[d]^{\f{M}(f,\bb{F}[x])}\\
B\ar[r]_-{\Upsilon_{B,\bb{F}[x]}}&\bb{F}[x]\ot\f{M}_{B,\bb{F}[x]}\ar@<2.5pt>[r]^-{x\mapsto0}\ar@<-2.5pt>[r]_-{x\mapsto1}&
\f{M}_{B,\bb{F}[x]}}} \\
    \end{array}
  \right)
$$
Thus $\lceil F\rceil$ may be considered as a functor. Suppose that $H:A\to B[x]$ is an elementary homotopy from $f$ to $g$.
Let $\ov{H}:\f{M}_{A,\bb{F}[x]}\to B$ denote the associated pro-morphism of $H$. We have
$$f=\ov{H}\big((x\mapsto0)\ot\r{id}\big)\Upsilon_{A,\bb{F}[x]}\hspace{3mm}\text{and}\hspace{3mm}
g=\ov{H}\big((x\mapsto1)\ot\r{id}\big)\Upsilon_{A,\bb{F}[x]}.$$ It follows that,
\begin{equation*}\begin{split}\big(\underleftarrow{\lim}F(f)\big)\alpha^F(A)&=\big(\underleftarrow{\lim}F(\ov{H})\big)\hat{\f{p}}_0\alpha^F(A)\\
&=\big(\underleftarrow{\lim}F(\ov{H})\big)\hat{\f{p}}_1\alpha^F(A)\\&=\big(\underleftarrow{\lim}F(g)\big)\alpha^F(A).\end{split}\end{equation*}
This implies that $\lceil F\rceil(f)=\lceil F\rceil(g)$, and hence $\lceil F\rceil$ is homotopy invariant.
Suppose $G:\b{A}\to\b{C}$ is homotopy invariant and $\beta:G\to F$ is a natural transformation. We have the following commutative diagram
in $\r{pro}\text{-}\b{C}$:$$\xymatrix{F(A)\ar[rr]^-{F(\Upsilon_{A,\bb{F}[x]})}&&F(\bb{F}[x]\ot\f{M}_{A,\bb{F}[x]})
\ar@<2.5pt>[r]^-{F(x\mapsto0)}\ar@<-2.5pt>[r]_-{F(x\mapsto1)}&F(\f{M}_{A,\bb{F}[x]})\\
G(A)\ar[rr]_-{G(\Upsilon_{A,\bb{F}[x]})}\ar[u]^-{\beta(A)}&&G(\bb{F}[x]\ot\f{M}_{A,\bb{F}[x]})
\ar@<2.5pt>[r]^-{G(x\mapsto0)}\ar@<-2.5pt>[r]_-{G(x\mapsto1)}&G(\f{M}_{A,\bb{F}[x]})\ar[u]_{\beta(\f{M}_{A,\bb{F}[x]})}}$$
On the other hand, by Lemma \ref{2010111932}(iii), the pro-morphisms
$$\xymatrix{G(\bb{F}[x]\ot\f{M}_{A,\bb{F}[x]})\ar@<2.5pt>[r]^-{G(x\mapsto0)}\ar@<-2.5pt>[r]_-{G(x\mapsto1)}&G(\f{M}_{A,\bb{F}[x]})}$$
are (componentwise) equal. Hence we have $$\hat{\f{p}}_0\beta(A)=\hat{\f{p}}_1\beta(A).$$
Now, it follows from the universal property of the equalizer $\lceil F\rceil(A)$ that there is a
unique morphism $\ov{\beta}(A):G(A)\to\lceil F\rceil(A)$ satisfying $\beta(A)=\alpha^F(A)\ov{\beta}(A)$.
\end{proof}
The following is seen directly from Definition \ref{2010231857} and the proof of Theorem \ref{2010141245}.
\begin{theorem}
\emph{The functor $\f{P}^*:\b{A}_*\to\b{A}_*$ is a costrict homotopization of the identity functor $\r{id}:\b{A}_*\to\b{A}_*$.
Therefore, in a sense, $\f{P}^*$ is the greatest homotopy invariant subfunctor of $\r{id}$.}
\end{theorem}
More generally, the following result is also seen from the above theorem and the proof of Theorem \ref{2010141245}.
\begin{theorem}
\emph{Let $\b{C}$ be a category with arbitrary limits and $F:\b{A}_*\to\b{C}$ be a functor that preserves arbitrary limits.
Then $\lceil F\rceil=F\circ\f{P}^*$.}
\end{theorem}
As a direct corollary of the above theorem we have,
\begin{corollary}
\emph{Let $S$ denote an ordinary scheme over $\bb{F}$ regarded as a Zariski-sheaf $\b{A}_\r{cu}\to\b{Set}$ \cite{DemazureGabriel1}.
Then $\lceil S\rceil=S\circ\f{P}^\r{cu}$.}
\end{corollary}
In particular in case $S$ is an affine scheme regarded as the functor (\ref{2012082241}) we have
\begin{equation}\label{2012082244}
\lceil S\rceil=\r{Hom}_{\b{A}_\r{cu}}(A,\f{P}^\r{cu}(?)).
\end{equation}
Thus for any affine scheme $S$ we have defined three types $\pi_0(S),\lfloor S\rfloor,\lceil S\rceil$ of homotopizations given by
(\ref{2012082236}),(\ref{2012082231}),(\ref{2012082244}).
\begin{remark}
\emph{It is clear that the concept of (co)strict homotopization can be defined for any functor $F:\b{B}\to\b{C}$ where $\b{B}$ is a category
with a compatible relation as defined in $\S$\ref{1910191623}, or more generally, for any category $\b{B}$ with a class of weak equivalences.
Also it seems that the analogues of Theorems \ref{2012081413} and \ref{2010141245} can be stated for model categories $B$
with natural cylinder and path objects \cite[$\S$I.3 and $\S$I.3a]{Baues1}.}
\end{remark}
%%%%%%%%%%%%%%%%%%%%%%%%%%%%%%%%%%%%%%%%%%%%%%%%%%%%%%%%%%%%%%%%%%%%%%%%%%%%%%%%%%%%%%%%%%%%%%%%%%%%%%%%%%%%%%%%%%%%%%%%%%
%%%%%%%%%%%%%%%%%%%%%%%%%%%%%%%%%%%%%%%%%%%%%%%%%%%%%%%%%%%%%%%%%%%%%%%%%%%%%%%%%%%%%%%%%%%%%%%%%%%%%%%%%%%%%%%%%%%%%%%%%%
%%%%%%%%%%%%%%%%%%%%%%%%%%%%%%%%%%%%%%%%%%%%%%%%%%%%%%%%%%%%%%%%%%%%%%%%%%%%%%%%%%%%%%%%%%%%%%%%%%%%%%%%%%%%%%%%%%%%%%%%%%
%%%%%%%%%%%%%%%%%%%%%%%%%%%%%%%%%%%%%%%%%%%%%%%%%%%%%%%%%%%%%%%%%%%%%%%%%%%%%%%%%%%%%%%%%%%%%%%%%%%%%%%%%%%%%%%%%%%%%%%%%%
\section{Intrinsic Singular Cohomology}\label{2009232252}
Let $\Delta:=\{\b{n}\}$ denote the category of finite ordinals $\b{n}:=\{0<\ldots<n\}$. The standard simplicial algebra
\cite{CortinasThom1,Garkusha1} $\bb{F}[\Delta]\in\r{sim}\text{-}\b{A}_\r{cu}$ is given by the following assignments:
$$\b{n}\mapsto\bb{F}[\Delta_n]:=\bb{F}[x_0,\ldots,x_n]/\langle1-\sum_{i=0}^nx_i\rangle,\hspace{7mm}
(\alpha:\b{n}\to\b{m})\mapsto([x_i]\mapsto\sum_{j\in\alpha^{-1}(i)}[x_j]).$$
For $*\in\{\r{nc,c,cr,u,cu,cur}\}$, consider the functors
$$\f{S}^*:=\f{M}^*(?,\bb{F}[\Delta])\hspace{3mm}\text{and}\hspace{3mm}\underleftarrow{\f{S}}^*:=\underleftarrow{\lim}\f{M}^*(?,\bb{F}[\Delta])$$
given by:
\begin{equation*}%\label{2009231041}
\f{S}^*:\b{A}_\r{nc}/\b{A}_\r{u}\to\r{cosim}\text{-}\r{pro}\text{-}\b{A}_*\hspace{5mm}
A\mapsto\big(\b{n}\mapsto\f{M}^*(A,\bb{F}[\Delta_n])\big)
\end{equation*}
\begin{equation*}%\label{2009231042}
\underleftarrow{\f{S}}^*:\b{A}_\r{nc}/\b{A}_\r{u}\to\r{cosim}\text{-}\b{A}_*\hspace{5mm}
A\mapsto\big(\b{n}\mapsto\underleftarrow{\lim}\f{M}^*(A,\bb{F}[\Delta_n])\big)
\end{equation*}
Let $F:\b{A}_*\to\b{Ab}$ be an arbitrary functor. For any $A\in\b{A}_\r{nc}/\b{A}_\r{u}$, we have the cosimplicial abelian groups
$F\underleftarrow{\f{S}}^*(A)$ and $\underleftarrow{\lim}F\f{S}^*(A)$.
Let $\f{N}:\r{cosim}\text{-}\b{Ab}\to\b{coChain}$ denote the normalization functor. We let (see \cite[$\S$ 8.4]{Weibel2}):
$$\ov{\c{H}}^n_{F}(A):=\r{H}^n\f{N}F\underleftarrow{\f{S}}^*(A)\cong\pi^n(F\underleftarrow{\f{S}}^*(A),0)$$
$$\underline{\c{H}}^n_{F}(A):=\r{H}^n\f{N}\underleftarrow{\lim}F\f{S}^*(A)\cong\pi^n(\underleftarrow{\lim}F\f{S}^*(A),0)$$
It is clear that $\ov{\c{H}}^n_{F}$ and $\underline{\c{H}}^n_{F}$ may be considered as the functors
$$\ov{\c{H}}^n_{F},\underline{\c{H}}^n_{F}:\b{A}_\r{nc}/\b{A}_\r{u}\to\b{Ab}.$$
We have a canonical natural transformation $\underleftarrow{\f{S}}^*\to\f{S}^*$ induced by the inverse limits.
This, in turn,  induces the canonical natural transformations
$$F\underleftarrow{\f{S}}^*\to\underleftarrow{\lim}F\f{S}^*\hspace{3mm}\text{and (hence)}\hspace{3mm}\ov{\c{H}}^n_{F}\to\underline{\c{H}}^n_{F}.$$
\begin{theorem}\label{2011101935}
\emph{The functors $\ov{\c{H}}^n_{F}$ and $\underline{\c{H}}^n_{F}$ are homotopy invariant.}
\end{theorem}
The proof is an adaptation of the usual proof of homotopy invariance of singular cohomology of topological spaces by using prism operators.
See \cite[Theorem 2.10]{Hatcher1}.
\begin{proof}
For every $0\leq i\leq n$, let $\phi_i^n:\bb{F}[\Delta_n]\ot\bb{F}[x]\to\bb{F}[\Delta_{n+1}]$ be defined by
$[x_j]\mapsto[x_j]$ if $j<i$; $[x_i]\mapsto[x_i+x_{i+1}]$; $[x_j]\mapsto[x_{j+1}]$ if $j>i$; and $x\mapsto[x_{i+1}+\cdots+x_{n+1}]$.
For every algebra $B$ let
$$\Gamma_n:\f{M}(B\ot\bb{F}[x],\bb{F}[\Delta_n]\ot\bb{F}[x])\to\f{M}(B,\bb{F}[\Delta_n])$$
be the unique pro-morphism satisfying
$$(\r{id}_{\bb{F}[\Delta_n]}\ot\r{flip})(\Upsilon_{B,\bb{F}[\Delta_n]}\ot\r{id}_{\bb{F}[x]})=
(\r{id}_{\bb{F}[\Delta_n]}\ot\r{id}_{\bb{F}[x]}\ot\Gamma_n)\Upsilon_{B\ot\bb{F}[x],\bb{F}[\Delta_n]\ot\bb{F}[x]}.$$
Suppose that $H:A\to B\ot\bb{F}[x]$ is an elementary homotopy from $f$ to $g$. Let
$$\hat{\phi}^n_i:\underleftarrow{\lim}\f{M}(A,\bb{F}[\Delta_{n+1}])\to\underleftarrow{\lim}\f{M}(B,\bb{F}[\Delta_n])$$
be given by
$$\hat{\phi}^n_i:=\underleftarrow{\lim}\big(\Gamma_n\f{M}(H,\bb{F}[\Delta_n]\ot\bb{F}[x])\f{M}(A,\phi^n_i)\big),$$
and let the prism operator
$$P^{n}:F\big(\underleftarrow{\lim}\f{M}(A,\bb{F}[\Delta_{n+1}])\big)\to F\big(\underleftarrow{\lim}\f{M}(B,\bb{F}[\Delta_n])\big)$$
be given by $P^{n}:=\sum_{i=0}^n(-1)^iF(\hat{\phi}^n_i)$. Then, it can be checked that
$(P^n)_n$ is a cochain homotopy between
$\f{N}F\underleftarrow{\f{S}}(f)$ and $\f{N}F\underleftarrow{\f{S}}(g)$, and hence $\ov{\c{H}}^n_{F}(f)=\ov{\c{H}}^n_{F}(g)$.
Thus $\ov{\c{H}}^n_{F}$ is homotopy invariant. Homotopy invariance of $\underline{\c{H}}^n_{F}$ is obtained similarly.
\end{proof}
\begin{theorem}\label{2011101934}
\emph{We have the following natural isomorphism of abelian groups: $$\ov{\c{H}}^0_{F}(A)\cong\lceil F\rceil(A)\hspace{10mm}(A\in\b{A}_*)$$}
\end{theorem}
\begin{proof}It follows from the proof of Theorem \ref{2010141245} and definition of $\ov{\c{H}}^0_{F}(A)$.
(Note that we are required to have $A\cong\f{M}^*_{A,\bb{F}[\Delta_0]}$. Thus $A$ must be in $\b{A}_*$.)\end{proof}
So, the restricted functors $$\ov{\c{H}}^n_{F}:\b{A}_*\to\b{Ab}$$ may be
interpreted as a sort of \emph{higher} costrict homotopizations of $F:\b{A}_*\to\b{Ab}$.

In case $F$ is the identity functor $\r{id}:\b{A}_*\to\b{A}_*$, we put $$\c{H}^n_*(A):=\ov{\c{H}}^n_{F}(A)=\underline{\c{H}}^n_{F}(A).$$
Thus we have defined the homotopy invariant functors$$\c{H}^n_\r{nc},\c{H}^n_\r{c},\c{H}^n_\r{cr}:\b{A}_\r{nc}\to\b{Vec}\hspace{3mm}\text{and}
\hspace{3mm}\c{H}^n_\r{u},\c{H}^n_\r{cu},\c{H}^n_\r{cur}:\b{A}_\r{u}\to\b{Vec}.$$
From (\ref{2011111033}) we have the natural transformations
$$\c{H}^n_\r{nc}\to\c{H}^n_\r{c}\to\c{H}^n_\r{cr}\hspace{3mm}\text{and}\hspace{3mm}\c{H}^n_\r{u}\to\c{H}^n_\r{cu}\to\c{H}^n_\r{cur}.$$
For any algebra $A$ it can be easily checked that
\begin{equation}\label{2011112226}
\c{H}^0_\r{nc}(A)=\f{P}^\r{nc}(A)=0,\hspace{1mm}\c{H}^0_\r{c}(A)=\f{P}^\r{c}(A_\r{com}),\hspace{1mm}
\c{H}^0_\r{cr}(A)=\f{P}^\r{cr}((A_\r{com})_\r{red})
\end{equation}
Analogous identities are satisfied for unital cases. By Theorem \ref{2010312353} we have
\begin{equation}\label{2011112227}
\c{H}^0_\r{c}(A)=\c{H}^0_\r{cu}(A),\hspace{2mm}\c{H}^0_\r{cr}(A)=\c{H}^0_\r{cur}(A)\hspace{5mm}(A\in\b{A}_\r{u}).
\end{equation}
We now define a cup product between the cocycles:\\
Let the morphisms $\phi_{n,m},\psi_{n,m}$ be given by
$$\phi_{n,m}:\bb{F}[\Delta_{n+m}]\to\bb{F}[\Delta_{n}],\hspace{2mm}[x_i]\mapsto[x_i]\hspace{1mm}\text{if}
\hspace{1mm}0\leq i\leq n\hspace{1mm}\text{and}\hspace{1mm}[x_i]\mapsto0\hspace{1mm}\text{otherwise}$$
$$\psi_{n,m}:\bb{F}[\Delta_{n+m}]\to\bb{F}[\Delta_{m}],\hspace{2mm}[x_{n+i}]\mapsto[x_{i}]\hspace{1mm}\text{if}
\hspace{1mm}0\leq i\leq m\hspace{1mm}\text{and}\hspace{1mm}[x_i]\mapsto0\hspace{1mm}\text{otherwise}$$
We let the cup product of the cochains $$c_n\in\underleftarrow{\lim}\f{M}^*_{A,\bb{F}[\Delta_n]}\hspace{3mm}\text{and}\hspace{3mm} c_m\in\underleftarrow{\lim}\f{M}^*_{A,\bb{F}[\Delta_m]}$$ be given by
$$c_n\smallsmile c_m:=\big(\underleftarrow{\lim}\f{M}^*_{A,\phi_{n,m}}(c_n)\big)\big(\underleftarrow{\lim}\f{M}^*_{A,\psi_{n,m}}(c_m)\big)
\in\underleftarrow{\lim}\f{M}^*_{A,\bb{F}[\Delta_{n+m}]}.$$
It can be checked that $\smallsmile$ makes the graded vector space $\oplus_{n=0}^\infty\underleftarrow{\lim}\f{M}^*_{A,\bb{F}[\Delta_n]}$
to a DGCA. Hence, $\smallsmile$ can be considered as a multiplication on
$$\c{H}_*(A):=\oplus_{n=0}^\infty\c{H}_*^n(A),$$
making it to a GCA.

For any affine scheme $S=\c{S}(A)$ ($A\in\b{A}_\r{cu}$) the \emph{intrinsic singular cohomology algebra} of $S$ is defined to be the GCA given by
$$\c{H}_\r{sing}(S)=\oplus_{n=0}^\infty\c{H}_\r{sing}^n(S):=\c{H}_\r{cu}(A).$$
Thus we have defined a $\bb{A}^1$-homotopy invariant cohomology theory
$$\c{H}_\r{sing}:\b{Aff}^\r{op}\to\b{A}_\r{u}.$$
We end this section by the following result.
\begin{theorem}
\emph{\textbf{de Rham Theorem at degree zero.} Suppose $\r{Char}(\bb{F})=0$. For any affine scheme $S$ the graded-algebra morphism
$\c{H}_\r{deR}(S)\to\c{H}_\r{sing}(S)$ induced by the universal property of de Rham complex is an isomorphism at degree zero:
$$\c{H}^0_\r{deR}(S)\cong\c{H}^0_\r{sing}(S).$$}
\end{theorem}
\begin{proof}
It follows from Theorem \ref{2010311847}, and identities (\ref{2011112226}) and (\ref{2011112227}).
\end{proof}
%%%%%%%%%%%%%%%%%%%%%%%%%%%%%%%%%%%%%%%%%%%%%%%%%%%%%%%%%%%%%%%%%%%%%%%%%%%%%%%%%%%%%%%%%%%%%%%%%%%%%%%%%%%%%%%%%%%%%%%%%%
%%%%%%%%%%%%%%%%%%%%%%%%%%%%%%%%%%%%%%%%%%%%%%%%%%%%%%%%%%%%%%%%%%%%%%%%%%%%%%%%%%%%%%%%%%%%%%%%%%%%%%%%%%%%%%%%%%%%%%%%%%
%%%%%%%%%%%%%%%%%%%%%%%%%%%%%%%%%%%%%%%%%%%%%%%%%%%%%%%%%%%%%%%%%%%%%%%%%%%%%%%%%%%%%%%%%%%%%%%%%%%%%%%%%%%%%%%%%%%%%%%%%%
%%%%%%%%%%%%%%%%%%%%%%%%%%%%%%%%%%%%%%%%%%%%%%%%%%%%%%%%%%%%%%%%%%%%%%%%%%%%%%%%%%%%%%%%%%%%%%%%%%%%%%%%%%%%%%%%%%%%%%%%%%
\section{Generalized Morphisms Between Algebras}\label{2008311351}
For any algebra $A$ denote by $\f{I}(A)$ (resp. $\f{I}_\ell(A),\f{I}_\r{r}(A)$) the set of all ideals (resp. left, right ideals) of $A$.
Also, let $\f{I}_*(A)\subseteq\f{I}(A)$ where $*\in\{\r{prpr,fcodim,prm,max}\}$ denote respectively the subset of proper, proper with finite
codimension, prime, and maximal ideals of $A$. $\f{I}$ may be considered as a functor $\b{A}_\r{nc}^\r{op}\to\b{Set}$ in the obvious way.
Similarly, we have the functors $\f{I}_\ell,\f{I}_\r{r}:\b{A}_\r{nc}^\r{op}\to\b{Set}$,
$\f{I}_\r{prpr},\f{I}_\r{fcodim}:\b{A}_\r{u}^\r{op}\to\b{Set}$, and $\f{I}_\r{prm}:\b{A}_\r{cu}^\r{op}\to\b{Set}$, and in case $\bb{F}$ is
algebraically closed, the functor $\f{I}_\r{max}:\b{A}_\r{cufg}^\r{op}\to\b{Set}$. Thus by the canonical extension we have also the functors
$$\f{I},\f{I}_\ell,\f{I}_\r{r}:\r{pro}\text{-}\b{A}_\r{nc}^\r{op}\to\r{ind}\text{-}\b{Set},\hspace{2mm}
\f{I}_\r{prpr},\f{I}_\r{fcodim}:\r{pro}\text{-}\b{A}_\r{u}^\r{op}\to\r{ind}\text{-}\b{Set}$$
$$\f{I}_\r{prm}:\r{pro}\text{-}\b{A}_\r{cu}^\r{op}\to\r{ind}\text{-}\b{Set},\hspace{3mm}
\f{I}_\r{max}:\r{pro}\text{-}\b{A}_\r{cufg}^\r{op}\to\r{ind}\text{-}\b{Set}.$$
For $A\in\r{pro}\text{-}\b{A}_\r{nc}$ \emph{we call} any member of the set $\underrightarrow{\lim}\f{I}(A)$ a pro-ideal of $A$.
We need some facts and definitions: Let $A=(A_i,\phi_{i'i})$ be a pro-algebra. (i) Any pro-ideal $T$ of $A$ is a class $[T_i]$ of the equivalence
relation $\sim$ on $\dot{\cup}_i\f{I}(A_i)$ defined by $$\big(T_i\sim T_{i'}\big)\Leftrightarrow\big(\exists i''\geq i',i:\phi_{i''i}^{-1}(T_i)=\phi_{i''i'}^{-1}(T_{i'})\big)\hspace{2mm}\text{for}\hspace{2mm}T_i\in\f{I}(A_i),T_{i'}\in\f{I}(A_{i'}).$$ (ii) For any
pro-morphism $f$ from $A$ into an algebra $B$ we let $\ker(f)$ denote the pro-ideal $[\underrightarrow{\lim}\f{I}(f)](0)$ of $A$ where $0$ denotes the
zero-ideal of $B$. (Note that the notion of zero pro-ideal is meaningless.) (iii) For pro-ideals $T=[T_i],T'=[T'_{i'}]$ of $A$ we let the pro-ideal
$T\cap T'$ be defined by $[\phi_{i''i}^{-1}(T_i)\cap\phi_{i''i'}^{-1}(T'_{i'})]$ where $i''$ is an arbitrary index greater than $i$ and $i'$. It is
easily verified that the \emph{intersection} $T\cap T'$ is well-defined. We write $T\subseteq T'$ and say that $T'$ includes $T$ if there exists
$i''\geq i,i'$ such that $\phi_{i''i}^{-1}(T_i)\subseteq\phi_{i''i'}^{-1}(T'_{i'})$. It can be checked that \emph{inclusion} is a well-defined partial ordering on pro-ideals.\\
In (iv)-(vii) below suppose $A=(A_i,\phi_{i'i}),B=(B_j,\psi_{j'j})$ are pro-algebras such that all structural morphisms $\phi_{i'i},\psi_{j'j}$ are
surjective: (iv) If $T_i\in\f{I}(A_i)$ and $T_{i'}\in\f{I}(A_{i'})$ represent the same pro-ideal $T$ of $A$, then $$A_i/T_i\cong
A_{i''}/\phi_{i''i}^{-1}(T_i)=A_{i''}/\phi_{i''i'}^{-1}(T_{i'})\cong A_{i'}/T_{i'}\hspace{5mm}(\text{for some}\hspace{2mm}i''\geq i,i').$$
Thus we can define the associated \emph{quotient algebra} up to isomorphism by $A/T:=A_i/T_i$.
Note that we have a canonical quotient pro-morphism $\r{q}_T:A\to A/T$ represented by the quotient morphism $\r{q}_{T_i}:A_i\to A_i/T_i$.
Also, $T=\ker(\r{q}_T)$. (v) It is clear that if for every indexes $i,i'$ the algebras $A_i,A_{i'}$ are considered as ideals in themselves then
$[A_i]=[A_{i'}]$. We denote the pro-ideal $[A_i]$ of $A$ by $A_\r{ideal}$. (vi) For pro-ideals $T=[T_i],T'=[T'_{i'}]$ of $A$ we let the pro-ideal
$T+T'$ be defined by $[\phi_{i''i}^{-1}(T_i)+\phi_{i''i'}^{-1}(T'_{i'})]$ where $i''$  is an arbitrary index with $i''\geq i,i'$.
The \emph{sum} $T+T'$ is well-defined. (vii) For pro-ideals $T=[T_i],S=[S_j]$ of $A,B$ we associate a pro-ideal of $A\ot B$ given by
$$T\ot B_\r{ideal}+A_\r{ideal}\ot S:=[T_i\ot B_j+A_i\ot S_j]$$
It is easily checked that this is well-defined. ($T\ot S:=[T_i\ot S_j]$ is not well-defined.)

We define a category $\bar{\b{A}}_\r{nc}$ as follows: The objects of $\bar{\b{A}}_\r{nc}$ are those of $\b{A}_\r{nc}$, and
$$\r{Hom}_{\bar{\b{A}}_\r{nc}}(A,B):=\underrightarrow{\lim}\f{I}(\f{M}^\r{nc}_{A,B}).$$
The composition $\circ:\r{Hom}_{\bar{\b{A}}_\r{nc}}(B,C)\times\r{Hom}_{\bar{\b{A}}_\r{nc}}(A,B)\to\r{Hom}_{\bar{\b{A}}_\r{nc}}(A,C)$ is given by
$$\underrightarrow{\lim}\f{I}(\f{M}_{B,C})\times\underrightarrow{\lim}\f{I}(\f{M}_{A,B})\to
\underrightarrow{\lim}\f{I}(\f{M}_{B,C}\ot\f{M}_{A,B})\to\underrightarrow{\lim}\f{I}(\f{M}_{A,C})$$ where the first map is given by
\begin{equation*}\label{2009102230}
(T,S)\mapsto T\ot(\f{M}_{A,B})_\r{ideal}+(\f{M}_{B,C})_\r{ideal}\ot S\end{equation*}
and the second map is $\underrightarrow{\lim}\f{I}(\Phi_{A,B,C})$. Equivalently, we may let
\begin{equation*}\label{2009102231}
T\circ S:=\ker[(\r{q}_T\ot\r{q}_S)\Phi_{A,B,C}].\end{equation*}
Associativity of $\circ$ follows from (\ref{2009061331}). For any morphism $f:A\to B$ in $\b{A}_\r{nc}$ denote by $\ov{f}\in\r{Pnt}(\f{M}_{A,B})$ the
point associated with $f$. By Corollary \ref{2012011947} the assignment $f\mapsto\ker(\ov{f})$ defines an
embedding of $\r{Hom}_{\b{A}_\r{nc}}(A,B)$ into $\r{Hom}_{\bar{\b{A}}_\r{nc}}(A,B)$. We have
$$\r{id}_{\f{M}_{A,B}}=\bigg(\xymatrix{\f{M}_{A,B}\ar[r]^-{\Phi_{A,A,B}}&\f{M}_{A,B}\ot\f{M}_{A,A}\ar[rr]^-{\r{id}_{\f{M}_{A,B}}\ot\ov{\r{id}}_A}&&
\f{M}_{A,B}\ot\bb{F}\cong\f{M}_{A,B}}\bigg)$$ where $\ov{\r{id}}_A\in\r{Pnt}(\f{M}_{A,A})$ denotes the pro-morphism associated to $\r{id}_A$.
Thus for any pro-ideal $T$ of $\f{M}_{A,B}$, $T\circ\ker(\ov{\r{id}}_A)=T$, and similarly, $\ker(\ov{\r{id}}_B)\circ T=T$. For morphisms $f:A\to B$
and $g:B\to C$ in $\b{A}_\r{nc}$ the identity $(\ov{g}\ot\ov{f})\Phi_{A,B,C}=\ov{gf}$ implies that $\ker(\ov{g})\circ\ker(\ov{f})=\ker(\ov{gf})$.
Hence $\bar{\b{A}}_\r{nc}$ is a category containing $\b{A}_\r{nc}$.
Similarly, we define generalized categories of algebras listed below:
\begin{enumerate}
\item[$\bullet$] For $*\in\{\ell,\r{r}\}$: $\bar{\b{A}}_\r{nc}^*\supset\b{A}_\r{nc}$ with
$$\r{Obj}(\bar{\b{A}}_\r{nc}^*):=\r{Obj}(\b{A}_\r{nc})\hspace{2mm}\text{and}\hspace{2mm}
\r{Hom}_{\bar{\b{A}}_\r{nc}^*}(A,B):=\underrightarrow{\lim}\f{I}_*(\f{M}^\r{nc}_{A,B}).$$
\item[$\bullet$] For $*\in\{\r{prpr,fcodim}\}$: $\bar{\b{A}}_\r{u}^*\supset\b{A}_\r{u}$ with
$$\r{Obj}(\bar{\b{A}}_\r{u}^*):=\r{Obj}(\b{A}_\r{u})\hspace{2mm}\text{and}\hspace{2mm}
\r{Hom}_{\bar{\b{A}}_\r{u}^*}(A,B):=\underrightarrow{\lim}\f{I}_*(\f{M}^\r{u}_{A,B}).$$
\item[$\bullet$] $\bar{\b{A}}_\r{u}^\r{prm}\supset\b{A}_\r{u}$ with
$$\r{Obj}(\bar{\b{A}}_\r{u}^\r{prm}):=\r{Obj}(\b{A}_\r{u})\hspace{2mm}\text{and}\hspace{2mm}
\r{Hom}_{\bar{\b{A}}_\r{u}^\r{prm}}(A,B):=\underrightarrow{\lim}\f{I}_\r{prm}(\f{M}^\r{cu}_{A,B}).$$
\item[$\bullet$] In case $\bb{F}$ is algebraically closed: $\bar{\b{A}}_\r{ufg}^\r{max}\supset\b{A}_\r{ufg}$ with
$$\r{Obj}(\bar{\b{A}}_\r{ufg}^\r{max}):=\r{Obj}(\b{A}_\r{ufg})\hspace{2mm}\text{and}\hspace{2mm}
\r{Hom}_{\bar{\b{A}}_\r{ufg}^\r{max}}(A,B):=\underrightarrow{\lim}\f{I}_\r{max}(\f{M}^\r{cu}_{A,B}).$$\end{enumerate}
In the following theorem we record some additional properties of these categories.
\begin{theorem}
\emph{\begin{enumerate}
\item[(i)] $\bar{\b{A}}_\r{nc}^\ell$ and $\bar{\b{A}}_\r{nc}^\r{r}$ are equivalent categories.
\item[(ii)] $\bar{\b{A}}_\r{u}^\r{prpr}\supset\bar{\b{A}}_\r{u}^\r{fcodim}$. Also $\bar{\b{A}}_\r{u}^\r{prpr}$
is equivalent to a subcategory of $\bar{\b{A}}_\r{nc}$.
\item[(iii)] $\bar{\b{A}}_\r{u}^\r{prm}$ is equivalent to a subcategory of $\bar{\b{A}}_\r{u}^\r{prpr}$.
\item[(iv)] $\bar{\b{A}}_\r{ufg}^\r{max}$ and $\b{A}_\r{ufg}$ are equivalent ($\bb{F}$ is supposed to be algebraically closed).
\end{enumerate}}\end{theorem}
\begin{proof}(i) It follows from Theorem \ref{2009072327} that $\b{op}:\b{A}_\r{nc}\to\b{A}_\r{nc}$ extends to an equivalence between
$\bar{\b{A}}_\r{nc}^\ell$ and $\bar{\b{A}}_\r{nc}^\r{r}$. (ii) For every $A,B\in\b{A}_\r{u}$ let $\psi_{A,B}:\f{M}^\r{nc}_{A,B}\to\f{M}^\r{u}_{A,B}$
denote the pro-morphism satisfying $\Upsilon^\r{u}_{A,B}=(\r{id}_B\ot\psi_{A,B})\Upsilon^\r{nc}_{A,B}$. The functor given by $A\mapsto A$ and
$$\xymatrix{\underrightarrow{\lim}\f{I}(\f{M}^\r{u}_{A,B})\supset\underrightarrow{\lim}\f{I}_\r{prpr}(\f{M}^\r{u}_{A,B})
\ar[rr]^--{\underrightarrow{\lim}\f{I}(\psi_{A,B})}&&\underrightarrow{\lim}\f{I}(\f{M}^\r{nc}_{A,B})}$$
embeds $\bar{\b{A}}_\r{u}^\r{prpr}$ into $\bar{\b{A}}_\r{nc}$. The proof of (iii) is similar to that of (ii).
(iv) We know that $\b{A}_\r{ufg}\subset\bar{\b{A}}_\r{ufg}^\r{max}$. Let $A,B\in\b{A}_\r{ufg}$. By Theorem \ref{2009080009}(iii),
all components of $\f{M}^\r{cu}_{A,B}$ are in $\b{A}_\r{cufg}$. Thus if $T$ is a maximal pro-ideal of $\f{M}^\r{cu}_{A,B}$ then
$\f{M}^\r{cu}_{A,B}/T\cong\bb{F}$ and hence the canonical quotient pro-morphism $q:\f{M}^\r{cu}_{A,B}\to\f{M}^\r{cu}_{A,B}/T$ can be considered
as a member of $\r{Pnt}_{\neq0}(\f{M}^\r{cu}_{A,B})$. The proof is complete.\end{proof}
We now consider some properties of the extended categories of algebras.
\begin{theorem}
\emph{Intersection (resp. sum) and inclusion of pro-ideals induce enriched category structure on $\bar{\b{A}}_\r{nc}$ by
partially ordered idempotent semigroups.}\end{theorem}
\begin{proof}It is clear that $\cap$ and $\subseteq$ make $\r{Hom}_{\bar{\b{A}}_\r{nc}}(A,B)$ into a partially ordered idempotent-semigroup. Suppose
that $S=[S_i],S'=[S'_{i'}]$ are pro-ideals of $\f{M}_{A,B}$ and $T=[T_j]$ be a pro-ideal of $\f{M}_{B,C}$. Without lost of generality assume $i=i'$.
We have \begin{equation*}\begin{split}&(\f{M}_{B,C})_j\ot(S_i\cap S'_i)+T_j\ot(\f{M}_{A,B})_i\\
=&\big((\f{M}_{B,C})_j\ot S_i+T_j\ot(\f{M}_{A,B})_i\big)\cap\big((\f{M}_{B,C})_j\ot S'_i+T_j\ot(\f{M}_{A,B})_i\big)\end{split}\end{equation*}
in the algebra $(\f{M}_{B,C})_j\ot(\f{M}_{A,B})_i$. It follows that $T\circ(S\cap S')=(T\circ S)\cap(T\circ S')$. Similarly, $(T\cap T')\circ S=
(T\circ S)\cap(T'\circ S)$. Thus $\circ$ is a bihomomorphism with respect to the semigroup operation $\cap$. Also, it is clear that if $T\subseteq T',S\subseteq S'$ then $T\circ S\subseteq T'\circ S'$. The sum case is similar.\end{proof}
\begin{theorem}\label{2009092223}
\emph{Suppose that $A$ and $B$ are isomorphic in $\bar{\b{A}}_\r{u}^\r{prpr}$. Then they are also isomorphic in $\b{A}_\r{u}$.}\end{theorem}
We shall need the following elementary fact: Let $\phi:V\to W\ot V',\psi:W\to V\ot W'$ be morphisms in $\b{Vec}$. Suppose for some
$x\in V',y\in W'$ we have that $(\psi\ot\r{id})\phi(v)=v\ot y\ot x$ and $(\phi\ot\r{id})\psi(w)=w\ot x\ot y$.
Then there are isomorphisms $\hat{\phi}:V\to W$ and $\hat{\psi}:W\to V$ such that $\phi(v)=\hat{\phi}(v)\ot x$ and $\psi(w)=\hat{\psi}(w)\ot y$.
\begin{proof}Let $S,T$ be proper pro-ideals respectively in $\f{M}^\r{u}_{A,B},\f{M}^\r{u}_{B,A}$ such that
\begin{equation}\label{2009100927}
\ker[(\r{q}_T\ot\r{q}_S)\Phi_{A,B,A}]=\ker(\ov{\r{id}}_A)\hspace{5mm}\ker[(\r{q}_S\ot\r{q}_T)\Phi_{B,A,B}]=\ker(\ov{\r{id}}_B)\end{equation}
Let $\phi:=(\r{id}_B\ot\r{q}_S)\Upsilon^\r{u}_{A,B}$ and $\psi:=(\r{id}_A\ot\r{q}_T)\Upsilon^\r{u}_{B,A}$. Since
$\f{M}^\r{u}_{A,A}/\ker(\ov{\r{id}}_A)\cong\bb{F}$, the left side of (\ref{2009100927}) implies that image of $(\r{q}_T\ot\r{q}_S)\Phi_{A,B,A}$ is
the subalgebra generated by $1_{\f{M}^\r{u}_{B,A}/T}\ot1_{\f{M}^\r{u}_{A,B}/S}$. Thus, it follows from the identity
$$(\psi\ot\r{id})\phi=[\r{id}_A\ot((\r{q}_T\ot\r{q}_S)\Phi_{A,B,A})]\Upsilon^\r{u}_{A,A}$$
that $(\psi\ot\r{id})\phi(a)=a\ot1_{\f{M}^\r{u}_{B,A}/T}\ot1_{\f{M}^\r{u}_{A,B}/S}$. Similarly, it is shown that
$(\phi\ot\r{id})\psi(b)=b\ot1_{\f{M}^\r{u}_{A,B}/S}\ot1_{\f{M}^\r{u}_{B,A}/T}$. Now, the desired result follows from the mentioned fact.\end{proof}
Every morphism $f:A\to B$ in $\b{A}_\r{u}$ induces a fundamental functor $\hat{f}$ between categories
of unital left modules of $A$ and $B$ given by $$\hat{f}:\r{Mod}(A)\to\r{Mod}(B)\hspace{10mm}\hat{f}(M):=B\ot_AM$$
where $B$ is considered as a right $A$-module with module multiplication $b\cdot a:=bf(a)$.
The assignment $f\mapsto\hat{f}$ is simply extended to morphisms in $\bar{\b{A}}_\r{u}^\r{prpr}$ as follows:
Let $S$ be a proper pro-ideal of $\f{M}^\r{u}_{A,B}$ and let $\phi:A\to B\ot\f{M}_{A,B}/S$ be as in the proof of Theorem \ref{2009092223}.
$\hat{S}(M)$ is defined to be the $B$-submodule of $B\ot\f{M}_{A,B}/S$ with the underlying vector space $\hat{\phi}(M)$.
The action of $\hat{S}$ on morphisms of $\r{Mod}(A)$ is defined to be that of $\hat{\phi}$.
\begin{theorem}
\emph{The K-group functor ${K_0}$ on $\b{A}_\r{u}$ extends canonically to a functor $${K_0}:\bar{\b{A}}_\r{u}^\r{fcodim}\to\b{Ab}.$$}
\end{theorem}\begin{proof}
Let $S$ be a proper finite-codimensional pro-ideal of $\f{M}_{A,B}^\r{u}$. We show that $\hat{S}(M)$ is a finitely-generated projective $B$-module
for any finitely-generated projective $A$-module $M$: Let $C=\f{M}^\r{u}_{A,B}/S, D=B\ot C, N=D\ot_AM$.
We know that $C$ is a finite-dimensional algebra and $N$ as a $D$-module is finitely-generated and projective.
Suppose that $\{c_i\}_{i=1}^k$ is a vector basis for $C$. If $\{x_j\}^\ell_{j=1}$ generate $N$ as a $D$-module then $\{(1_B\ot c_i)x_j\}$ generates
$N$ as a $B$-module. Thus $N$ is a finitely generated $B$-module. We know that there exists a $D$-module isomorphism
$\Gamma:N\oplus N'\to\oplus_{i=1}^nD$ for some $n\geq1$ and some $D$-module $N'$. It is clear that $\Gamma$ is also a $B$-module isomorphism.
On the other hand, $D$ is a free $B$-module with $B$-basis $\{1_B\ot c_i\}$. Thus $\oplus_{i=1}^nD$ is also a free $B$-module. So, $N$ as a $B$-module
is projective. ${K_0}(S)$ is defined to be the group-morphism ${K_0}(A)\to{K_0}(B)$ induced by the assignment $M\mapsto\hat{S}(M)$
on finitely-generated projective modules.\end{proof}
\begin{theorem}
\emph{Products exist in $\bar{\b{A}}_\r{u}^\r{prm}$ and (hence) coincide with products of $\b{A}_\r{u}$.}
\end{theorem}\begin{proof}
Let $B_1,B_2$ be unital algebras and let $p_1,p_2$ denote the canonical projections from $B_1\oplus B_2$ respectively onto $B_1,B_2$.
Let $S_1,S_2$ be morphisms in $\bar{\b{A}}_\r{u}^\r{prm}$ from $A$ respectively to $B_1,B_2$. By a result similar to Theorem \ref{1910101520}
we have the canonical isomorphism $\f{M}^\r{cu}_{A,B_1\oplus B_2}\cong\f{M}^\r{cu}_{A,B_1}\ot\f{M}^\r{cu}_{A,B_2}$. So, we may consider
$$I:=S_1\ot(\f{M}^\r{cu}_{A,B_2})_\r{ideal}+(\f{M}^\r{cu}_{A,B_1})_\r{ideal}\ot S_2$$
as a pro-ideal of $\f{M}^\r{cu}_{A,B_1\oplus B_2}$. Then $I$ is the only morphism in $\bar{\b{A}}_\r{u}^\r{prm}$ satisfying
$$S_1:=\ker(\ov{p}_1)\circ I\hspace{3mm}\text{and}\hspace{3mm}S_2:=\ker(\ov{p}_2)\circ I.$$\end{proof}
%%%%%%%%%%%%%%%%%%%%%%%%%%%%%%%%%%%%%%%%%%%%%%%%%%%%%%%%%%%%%%%%%%%%%%%%%%%%%%%%%%%%%%%%%%%%%%%%%%%%%%%%%%%%%%%%%%%%%%%%%%
%%%%%%%%%%%%%%%%%%%%%%%%%%%%%%%%%%%%%%%%%%%%%%%%%%%%%%%%%%%%%%%%%%%%%%%%%%%%%%%%%%%%%%%%%%%%%%%%%%%%%%%%%%%%%%%%%%%%%%%%%%
%%%%%%%%%%%%%%%%%%%%%%%%%%%%%%%%%%%%%%%%%%%%%%%%%%%%%%%%%%%%%%%%%%%%%%%%%%%%%%%%%%%%%%%%%%%%%%%%%%%%%%%%%%%%%%%%%%%%%%%%%%
%%%%%%%%%%%%%%%%%%%%%%%%%%%%%%%%%%%%%%%%%%%%%%%%%%%%%%%%%%%%%%%%%%%%%%%%%%%%%%%%%%%%%%%%%%%%%%%%%%%%%%%%%%%%%%%%%%%%%%%%%%
\section{Classifying Algebras for Algebraic K-Groups}\label{2011161438}
In this section we consider a bivariant K-theory $\f{QQ}$ which is a pure-algebraic version of Cuntz's interpretation \cite{Cuntz1}
of the Kasparov bivariant K-theory of C*-algebras \cite{Kasparov1}. Using the functor $\f{M}^\r{nc}$ and following a method introduced by Phillips
\cite{Phillips2} we prove the existence of \emph{classifying homotopy pro-algebras} for $\f{QQ}$, Corti\~{n}as-Thom's KK-groups
\cite{CortinasThom1}, and Weibel's homotopy K-groups \cite{Weibel3}.

In this section we denote by $\c{R}$ the matrix in $\r{M}_2(\bb{F}[x])$ given by
$$\c{R}:=\left(\begin{array}{cc}1-x^2 & x^3-2x \\x & 1-x^2 \\\end{array}\right).$$
Note that $\c{R}$ is invertible. For a story about $\c{R}$ in Algebraic Homotopy see \cite[$\S$3.4]{CortinasThom1}.
We begin with some well-known lemmas.
\begin{lemma}\label{2012152154}
\emph{For any $B\in\b{A}_\r{nc}$, the morphism  $\r{M}_2(B)\to\r{M}_2(B[x])$ given by $M\mapsto\c{R}^{-1}M\c{R}$
is an elementary homotopy between $\r{id}_{\r{M}_2(B)}$ and the morphism
$$\left(\begin{array}{cc}a&b\\c&d\\\end{array}\right)\mapsto\left(\begin{array}{cc}d&-c\\-b&a\\\end{array}\right).$$}
\end{lemma}
\begin{lemma}\label{2012032136}
\emph{Let $\alpha,\beta:A\to B$ be morphisms in $\b{A}_\r{nc}$. Then the morphisms
$$(a\mapsto\left(\begin{array}{cc}\alpha(a) & 0 \\0 & \beta(a) \\\end{array}\right))\hspace{2mm}\text{and}\hspace{2mm}
(a\mapsto\left(\begin{array}{cc}\beta(a) & 0 \\0 & \alpha(a) \\\end{array}\right))$$ from $A$ into $\r{M}_2(B)$ are elementary homotopic.}
\end{lemma}
\begin{proof}It follows directly from Lemma \ref{2012152154}.\end{proof}
\begin{lemma}\label{2012032233}
\emph{With the assumptions of Lemma \ref{2012032136}, suppose $B$ is an ideal of a unital algebra $C$, and there exists $c\in C$ such that
$\beta(a)=c\alpha(a)c^{-1}$. Then the morphisms $$(a\mapsto\left(\begin{array}{cc}\alpha(a) & 0 \\0 & 0
\\\end{array}\right))\hspace{2mm}\text{and}\hspace{2mm}(a\mapsto\left(\begin{array}{cc}\beta(a) & 0 \\0 & 0 \\\end{array}\right))$$
from $A$ into $\r{M}_2(B)$ are  homotopic.}
\end{lemma}
\begin{proof}By Lemma \ref{2012032136}, we have
$$\big(a\mapsto\left(\begin{array}{cc}\alpha(a) & 0 \\0 & 0 \\\end{array}\right)\big)\approx
\big(a\mapsto\left(\begin{array}{cc}0 & 0 \\0 & \alpha(a)\\\end{array}\right)\big).$$
Hence, the morphisms $$a\mapsto\left(\begin{array}{cc}1 & 0 \\0 & c \\\end{array}\right)\left(\begin{array}{cc}\alpha(a) & 0 \\0 & 0
\\\end{array}\right)\left(\begin{array}{cc}1 & 0 \\0 & c^{-1} \\\end{array}\right)=\left(\begin{array}{cc}\alpha(a) & 0 \\0 & 0 \\\end{array}\right)$$ $$a\mapsto\left(\begin{array}{cc}1 & 0 \\0 & c \\\end{array}\right)\left(\begin{array}{cc}0 & 0 \\0 & \alpha(a)
\\\end{array}\right)\left(\begin{array}{cc}1 & 0 \\0 & c^{-1} \\\end{array}\right)=\left(\begin{array}{cc}0& 0 \\0 & \beta(a) \\\end{array}\right)$$
from $A$ into $\r{M}_2(B)$ are elementary homotopic.
Now, applying Lemma \ref{2012032136} another time, we get the desired result.\end{proof}
\begin{lemma}\label{2012032241}
\emph{Let $\alpha_i:A\to\r{M}_{k_i}(B)$ ($i=1,\ldots,n$) be morphisms in $\b{A}_\r{nc}$. Suppose $\sigma$ denotes a permutation of $\{1,\ldots,n\}$.
Then the morphisms$$a\mapsto\left(\begin{array}{ccccc}\alpha_1(a)&0&\cdots & 0 & 0 \\0 & \alpha_2(a)&\cdots&0&0\\\vdots&\vdots&\ddots&\vdots&\vdots\\
0&0&\cdots&\alpha_n(a)&0\\0&0&\cdots&0&0\\\end{array}\right)$$ $$a\mapsto\left(\begin{array}{ccccc}\alpha_{\sigma(1)}(a)&0&\cdots & 0 & 0 \\0 &
\alpha_{\sigma(2)}(a)&\cdots&0&0\\\vdots&\vdots&\ddots&\vdots&\vdots\\0&0&\cdots&\alpha_{\sigma(n)}(a)&0\\0&0&\cdots&0&0\\\end{array}\right)$$
from $A$ into $\r{M}_{2k}(B)$, where $k:=\sum_i k_i$,  are homotopic.}
\end{lemma}
\begin{proof}It follows directly from Lemma \ref{2012032233}.\end{proof}
For any algebra $B$, let $\r{M}_\bullet(B)$ denote the ind-algebra indexed over $\bb{N}$ with components $\r{M}_n(B)$ and structural
morphisms $\r{M}_n(B)\to\r{M}_{n+1}(B)$ given by
\begin{equation}\label{2011211406}M\mapsto\left(\begin{array}{cc}M & 0 \\0 & 0 \\\end{array}\right).
\end{equation}For any $m,n$, consider the morphism\begin{equation}\label{2012041425}
\Gamma_{m,n}:\r{M}_m(B)\oplus\r{M}_n(B)\to\r{M}_{m+n}(B)\hspace{5mm}(M,N)\mapsto\left(\begin{array}{cc}M & 0 \\0 & N
\\\end{array}\right)\end{equation}
Then it follows from Lemma \ref{2012032241} that the image of the family $\{\Gamma_{m,n}\}$ in $\r{ind}\text{-}\r{Hot}(\b{A}_\r{nc})$
defines an ind-morphism $$\Gamma:\r{M}_\bullet(B)\oplus\r{M}_\bullet(B)\to\r{M}_\bullet(B).$$
(Note that $\{\Gamma_{m,n}\}$ does not define a morphism in $\r{ind}\text{-}\b{A}_\r{nc}$ except for $B=0$.)
It also follows from Lemma \ref{2012032241}
that $\r{M}_\bullet(B)$ is an abelian monoid in $\r{ind}\text{-}\r{Hot}(\b{A}_\r{nc})$ with comultiplication $\Gamma$ and null element given by the
ind-morphism $0\to\r{M}_\bullet(B)$. It is not hard to see that all the above statements hold if the algebra $B$ is replaced by any ind-algebra.
Also, we may consider $\r{M}_\bullet$ as the $\r{wh}$ preserving functor $$\r{M}_\bullet:\r{ind}\text{-}\b{A}_\r{nc}\to\r{ind}\text{-}\b{A}_\r{nc}.$$
The following result is obvious.
\begin{proposition}\label{2011211155}
\emph{For $A,B\in\r{ind}\text{-}\b{A}_\r{nc}$, the set $[A,\r{M}_\bullet(B)]$
has a canonical abelian monoid structure induced by $\Gamma$. Moreover, this structure is functorial in $A$ and $B$.}
\end{proposition}
\begin{lemma}\label{2011211105}
\emph{Suppose there exists a morphism $f:A\to A$ such that $\tilde{f}\r{h}0$ where
$$\tilde{f}:A\to\r{M}_2(A)\hspace{5mm}a\mapsto\left(\begin{array}{cc}f(a) & 0 \\0 & a \\\end{array}\right).$$
Then the monoid $[A,\r{M}_\bullet(B)]$ is a group.}
\end{lemma}
\begin{proof}
It is easily verified that for every ind-morphism $g:A\to\r{M}_\bullet(B)$, $g\circ f$ is an inverse for $g$ in $[A,\r{M}_\bullet(B)]$.
\end{proof}
Consider the coproduct $A\star A$ in $\b{A}_\r{nc}$.
For any $a\in A$, let $a_1,a_2$ denote the two copies of $a$ in $A\star A$.
The algebra $\r{q}A$, originally introduced by Cuntz, is defined to be the kernel of the codiagonal morphism
$A\star A\to A$ ($a_1,a_2\mapsto a$) \cite{Cuntz1},\cite[$\S$4.11]{CortinasThom1}.
The key property of $\r{q}A$ is that for any two morphisms $\alpha,\beta:A\to B$ if $I$ is an ideal of $B$ such that
$\alpha(a)-\beta(a)\in I$, then the restriction to $\r{q}A$ of the morphism $A\star A\to B$ given by $a_1\mapsto\alpha(a)$ and $a_2\mapsto\beta(a)$,
takes its values in $I$, and thus may be regarded as a morphism $\r{q}A\to I$.
We may consider $\r{q}$ as a functor $\r{q}:\b{A}_\r{nc}\to\b{A}_\r{nc}$ in the obvious way. Then, it
is easily verified that $\r{q}$ is homotopy preserving.
The proof of the next lemma is an adapted version of that of \cite[Proposition 1.15]{Phillips2}.
\begin{lemma}\label{2011201234}
\emph{$\r{q}A$ satisfies the assumption of Lemma \ref{2011211105}.}
\end{lemma}
\begin{proof}
Consider the morphisms $\alpha,\beta:A\to\r{M}_2(A\star A)[x]$ given by
$$\alpha:a\mapsto\left(\begin{array}{cc}a_2 & 0 \\0 & a_1 \\\end{array}\right)\hspace{5mm}\text{and}\hspace{5mm}
\beta:a\mapsto\c{R}^{-1}\left(\begin{array}{cc}a_1 & 0 \\0 & a_2 \\\end{array}\right)\c{R}.$$
It is easily verified that $\alpha(a)-\beta(a)\in\r{M}_2(\r{q}A)[x]$, and thus we have the canonical morphism
$\varphi:\r{q}A\to\r{M}_2(\r{q}A)[x]$ induced by $\alpha,\beta$.
Let $\r{s}:A\star A\to A\star A$ denote switch i.e. the morphism defined by $a_1\mapsto a_2$ and $a_2\mapsto a_1$.
Put $f:=s|_{\r{q}A}$. Then $\varphi$ is an elementary homotopy from $\tilde{f}$ to $0$. \end{proof}
The following result is a version of \cite[Proposition 1.4]{Cuntz1}.
\begin{theorem}\label{2012041236}
\emph{For any two algebras $A$ and $B$, the set $$\f{QQ}(A,B):=[\r{q}A,\r{M}_\bullet(B)]$$ has a canonical abelian group
structure. Moreover, we may consider the following homotopy invariant functor in the obvious way:
$$\f{QQ}:\b{A}_\r{nc}^\r{op}\times\b{A}_\r{nc}\to\b{Ab}.$$}
\end{theorem}
\begin{proof}It follows from Proposition \ref{2011211155}, and Lemmas \ref{2011211105} and \ref{2011201234}.\end{proof}
For any algebra $A$, let the \emph{evaluation} morphisms $\r{ev}_1,\r{ev}_2:\r{q}A\to A$ be defined respectively by $(a_1\mapsto a,a_2\mapsto0)$ and
$(a_1\mapsto0,a_2\mapsto a)$.
\begin{remark}
\emph{It is proved by Cuntz that for any C*-algebra $A$, $\r{q}A$ and $\r{q}^2A$ are analytically homotopic up to stabilization by $\r{M}_2$
\cite[Theorem 1.6]{Cuntz1}. Using this result, the triple $(\r{q}A,\r{ev}_1,\r{ev}_2)$ can be regarded as an analogue of the triple $([0,1],0,1)$,
homology theories of C*-algebras which are stable with respect to a tensor with compact operators.
Unfortunately, it seems that there is no way to restate a pure-algebraic version of this result.}
\end{remark}
We let $$\f{Q}(B):=\f{QQ}(\bb{F},B).$$ Thus we have the following homotopy invariant functor:
$$\f{Q}:\b{A}_\r{nc}\to\b{Ab}.$$
Let $C\in\b{A}_\r{u}$. Denote by $\r{Idmp}_n(C)$ the set of idempotent matrixes in $\r{M}_n(C)$. Then (\ref{2011211406}) makes the family $(\r{Idmp}_n(C))_{n\geq1}$ into an ind-set; write $\r{Idmp}(C)$ for its direct limit. The group
$\lim_{n\to\infty}\r{GL}_n(C)$ acts on $\r{Idmp}(C)$ by conjugation. Write $[\r{Idmp}(C)]$ for the set of conjugate classes.
The morphisms (\ref{2012041425}) induce an abelian monoid structure on $[\r{Idmp}(C)]$.
Its Grothendieck group is ${K_0}(C)$, the usual K-group of $C$.
It is clear that any $M\in\r{Idmp}_n(C)$ is exactly distinguished by the morphism $\alpha_M:\bb{F}\to\r{M}_n(C)$
given by $1\mapsto M$. By Lemma \ref{2012032233}, if the idempotents $M,M'\in\r{Idmp}_n(C)$ are conjugate then $\alpha_M,\alpha_{M'}$,
considered canonically as morphisms into $\r{M}_{2n}(C)$, are homotopic. Thus the assignment $M\mapsto\alpha_M$ induces a natural monoid-morphism
\begin{equation}\label{2012202312}
[\r{Idmp}(C)]\to[\bb{F},\r{M}_\bullet(C)].
\end{equation}
The morphism $\r{ev}_1:\r{q}\bb{F}\to\bb{F}$ induces the monoid-morphism
\begin{equation}\label{2012202318}
[\bb{F},\r{M}_\bullet(C)]\to[\r{q}\bb{F},\r{M}_\bullet(C)]=\f{Q}(C).
\end{equation}
The composition of (\ref{2012202312}) and (\ref{2012202318}) gives rise to the group-morphism
\begin{equation}\label{2012202322}
{K_0}(C)\to\f{Q}(C).
\end{equation}
We need the following Yoneda lemma.
\begin{lemma}\label{2012041916}
\emph{Let $\b{C}$ be a category with finite coproducts and cofinal object. Let $C\in\r{pro}\text{-}\b{C}$, and suppose that for
every $D\in\b{C}$ the set $\r{Hom}_{\r{pro}\text{-}\b{C}}(C,D)$ has an abelian group structure such that these structures make $\r{Hom}_{\r{pro}\text{-}\b{C}}(C,?)$ to a functor from $\b{C}$ to $\b{Ab}$. Then $C$ has a cocommutative cogroup structure in $\r{pro}\text{-}\b{C}$
whose induces the group structures of $\r{Hom}_{\r{pro}\text{-}\b{C}}(C,D)$ in the obvious way. }
\end{lemma}
\begin{theorem}\label{2011272145}
\emph{For any algebra $A$, there exists a pro-algebra $\ov{\f{QQ}}_A$ with a cocommutative cogroup structure
as an object in $\r{pro}\text{-}\r{Hot}(\b{A}_\r{nc})$
such that the abelian groups $[\ov{\f{QQ}}_A,B]$ and $\f{QQ}(A,B)$ are naturally isomorphic for every algebra $B$.
In particular, there exists a pro-algebra $\ov{\f{Q}}$ with a natural abelian group-morphism
$$[\ov{\f{Q}},B]\cong\f{Q}(B)\hspace{10mm}(B\in\b{A}_\r{nc}).$$}
\end{theorem}
\begin{proof}
By Theorem \ref{2012041948}, for every algebra $B$ we have
$$\f{QQ}(A,B)=[\r{q}A,\r{M}_\bullet(B)]=[\r{q}A,\r{M}_\bullet(\bb{F})\ot B]\cong[\f{M}^\r{nc}_{\r{q}A,\r{M}_\bullet(\bb{F})},B].$$
Thus it follows from Lemma \ref{2012041916} that $\f{M}^\r{nc}_{\r{q}A,\r{M}_\bullet(\bb{F})}$ is the desired pro-algebra
\end{proof}
Following Phillips \cite{Phillips2}, $\ov{\f{QQ}}_A$ may be called classifying pro-algebra.

For any algebra $B$, we put
$$\r{M}_\infty(B):=\lim_{n\to\infty}\r{M}_n(B)\hspace{4mm}\text{and}\hspace{4mm}\c{M}_\infty(B):=\r{M}_\bullet\r{M}_\infty(B).$$
It is clear that we have the $\r{wh}$-preserving functor
$$\c{M}_\infty:\r{ind}\text{-}\b{A}_\r{nc}\to\r{ind}\text{-}\b{A}_\r{nc}.$$
As explained in \cite[$\S$4.1]{CortinasThom1}, the abelian monoid structure on $\c{M}_\infty(B)$, may be induced also by the direct sum of
infinite matrixes in $\r{M}_\infty(B)$.
For any algebra $B$, we denote by $B[\Delta]$ the ind-algebra $B\ot\bb{F}[\Delta]$. For any simplicial set $S\in\r{sim}\text{-}\b{Set}$,
let $$\f{Z}(S,B):=\r{Hom}_{\r{sim}\text{-}\b{Set}}(S,B[\Delta])$$
where $B[\Delta]$, by the forgetful functor $\r{sim}\text{-}\b{A}_\r{nc}\to\r{sim}\text{-}\b{Set}$, is considered as a simplicial
set. Then $\f{Z}(S,B)$ can be considered as an algebra with pointwise operations. Thus we have the functor
$$\f{Z}:(\r{sim}\text{-}\b{Set})^\r{op}\times\b{A}_\r{nc}\to\b{A}_\r{nc}.$$
We extend $\f{Z}$ canonically to the functor
$$\f{Z}:(\r{pro}\text{-}\r{sim}\text{-}\b{Set})^\r{op}\times\b{A}_\r{nc}\to\r{ind}\text{-}\b{A}_\r{nc},$$
and then again to the functor
$$\f{Z}:(\r{pro}\text{-}\r{sim}\text{-}\b{Set})^\r{op}\times\r{ind}\text{-}\b{A}_\r{nc}\to\r{ind}\text{-}\b{A}_\r{nc}.$$
Let $\r{sd}:\r{sim}\text{-}\b{Set}\to\r{sim}\text{-}\b{Set}$ denote the simplicial subdivision functor \cite[$\S$III.4]{GoerssJardine1}.
For any $S\in\r{sim}\text{-}\b{Set}$ let $\r{sd}^\bullet S\in\r{pro}\text{-}\r{sim}\text{-}\b{Set}$ be given by the inverse system
$$S=\r{sd}^0 S\leftarrow\r{sd}^1 S\leftarrow\cdots\leftarrow\r{sd}^n S\leftarrow\cdots$$
where the structural morphisms are given by a natural transformation $\r{sd}\to\r{id}$ called last vertex map.
Let $\Delta^1$ denote standard $1$-simplex i.e. the simplicial set given by $$\b{n}\mapsto\r{Hom}_\Delta(\b{n},\b{1}).$$
For any ind-algebra $B$ let $B^{\c{S}^1}$ (\cite[$\S$3.3]{CortinasThom1}) denote the kernel of ind-morphism
$$\f{Z}(\r{sd}^\bullet\Delta^1,B)\to\f{Z}(\r{sd}^\bullet\partial\Delta^1,B)$$
induced by the canonical morphism $\partial\Delta^1\to\Delta^1$ in $\r{sim}\text{-}\b{Set}$. So we have the functor
$$?^{\c{S}^1}:\r{ind}\text{-}\b{A}_\r{nc}\to\r{ind}\text{-}\b{A}_\r{nc}$$
The $n$ times iteration of this functor is denoted by $?^{\c{S}^n}$. It is not hard to see that for any two algebras $A,B$
we have the following canonical isomorphisms of ind-algebras:
$$(A\ot B)^{\c{S}^n}\cong A^{\c{S}^n}\ot B\hspace{10mm}(\c{M}_\infty(B))^{\c{S}^n}\cong\c{M}_\infty(B^{\c{S}^n})$$
(When the ground ring is not a field, the proof of these latter identities follow from the nontrivial result \cite[Proposition 3.1.3]{CortinasThom1}.)
In particular, we have $$(B[x])^{\c{S}^n}\cong B^{\c{S}^n}[x],$$ and hence the functor $?^{\c{S}^n}$ is homotopy preserving.
By \cite[Theorem 3.3.2]{CortinasThom1} the set $$[A,B^{\c{S}^n}]\hspace{10mm}(A\in\b{A}_\r{nc},B\in\r{ind}\text{-}\b{A}_\r{nc})$$
has a canonical group structure, functorial in $A,B$, and abelian for $n\geq2$.

By an extension \cite[Definition 4.2.1]{CortinasThom1} we mean a sequence
\begin{equation}\label{2011270637}\xymatrix{C\ar[r]^{g}&B\ar[r]^{f}& A}\end{equation}
of ind-morphisms between ind-algebras such that $g=\ker(f)$ and $f=\r{coker}(g)$. For $B\in\r{ind}\text{-}\b{A}_\r{nc}$ the canonical morphism
$\Delta^0\to\Delta^1$ in $\r{sim}\text{-}\b{Set}$ induces an ind-morphism
$$\f{Z}(\r{sd}^\bullet\Delta^1,B)\to\f{Z}(\r{sd}^\bullet\Delta^0,B)\cong B.$$
Denote its kernel by $\c{P}(B)$. The canonical factorization $$\Delta^0\to\partial\Delta^1\to\Delta^1$$ of $\Delta^0\to\Delta^1$
gives us the so-called loop extension \cite[$\S$4.5]{CortinasThom1}:
\begin{equation}\label{2012211242}
B^{\c{S}^1}\to\c{P}(B)\to B.
\end{equation}
For $V\in\b{Vec}$ let $\r{T}(V):=\oplus_{n=1}^\infty V^{\ot^n}$ denote the tensor algebra associated to $V$. Then
$$\r{T}:\r{ind}\text{-}\b{Vec}\to\r{ind}\text{-}\b{A}_\r{nc}$$ is a left adjoint for the forgetful functor
$\r{ind}\text{-}\b{A}_\r{nc}\to\r{ind}\text{-}\b{Vec}$. For any ind-algebra $A$ let $\eta_A:\r{T}(A)\to A$
denote the adjoint of $\r{id}_A$ and put $\r{J}(A):=\ker\eta_A$. Let $\iota_A:\r{J}(A)\to\r{T}(A)$ denote the embedding.
Then we have the following extension:$$\xymatrix{\r{J}(A)\ar[r]^{\iota_A}&\r{T}(A)\ar[r]^{\eta_A}& A.}$$
It is not hard to see that $\r{J}$ can be considered as a homotopy preserving functor
$$\r{J}:\r{ind}\text{-}\b{A}_\r{nc}\to\r{ind}\text{-}\b{A}_\r{nc}.$$
Suppose the ind-morphism $h:A\to B$ in $\r{ind}\text{-}\b{Vec}$ is a splitting for extension (\ref{2011270637}) i.e. $fh=\r{id}_A$.
Let $\gamma_h$ denote the adjoint of $h$. Then there is a unique morphism $\xi_h$ making the following diagram commutative:
$$\xymatrix{\r{J}(A)\ar[r]^{\iota_A}\ar[d]^{\xi_h}&\r{T}(A)\ar[r]^{\eta_A}\ar[d]^{\gamma_h}& A\ar[d]^{\r{id}}\\C\ar[r]^{g}&B\ar[r]^{f}& A}$$
If $h':A\to B$ is another splitting for (\ref{2011270637}) then by \cite[Proposition 4.4.1]{CortinasThom1} we have $\xi_h\sim_\r{h}\xi_{h'}$.
So $\xi=\xi_h$ is just called a classifying map for (\ref{2011270637}). For more details on the constructions of $\r{J}(A),\c{P}(A),A^{\c{S}^n}$ see
\cite{CortinasThom1} or \cite{Garkusha3}.

For any two algebras $A,B$ we let
\begin{equation}\label{2012211555}
kk^{(n)}(A,B):=[\r{J}^n(A),(\c{M}_\infty(B))^{\c{S}^{n}}]=[\r{J}^n(A),\c{M}_\infty(B^{\c{S}^{n}})].
\end{equation}
It is clear that $kk^{(n)}$ may be considered as a homotopy invariant functor $$kk^{(n)}:\b{A}_\r{nc}^\r{op}\times\b{A}_\r{nc}\to\b{Ab}.$$
Let $f:A\to B$ be a morphism in $\b{A}_\r{nc}$. Then let $\r{j}(f)\in kk^{(1)}(A,B)$ be the composit
$$\xymatrix{\r{J}(A)\ar[r]^-{\r{J}(f)}&\r{J}(B)\ar[r]&\r{J}(\c{M}_\infty(B))
\ar[r]^-{\xi}&\c{M}_\infty(B^{\c{S}^{1}})},$$
where the second arrow is induced by the canonical embedding $B\to\c{M}_\infty(B)$ and where $\xi$ is a classifying map for the extention
induced by (\ref{2012211242}). Similarly, for any $n\geq1$ and any
homotopy class $[f]\in kk^{(n)}(A,B)$ of a ind-morphism $f$ let
$$\Phi^n_{A,B}([f])\in kk^{(n+1)}(A,B)$$ denote the homotopy class of the ind-morphism
defined by the composit $$\xymatrix{\r{J}^{n+1}(A)\ar[r]^-{\r{J}(f)}&\r{J}((\c{M}_\infty(B))^{\c{S}^{n}})
\ar[r]^-{\xi}&((\c{M}_\infty(B))^{\c{S}^{n}})^{\c{S}^{1}}\cong}(\c{M}_\infty(B))^{\c{S}^{n+1}}.$$
Then, as it is explained in \cite[$\S$6.1]{CortinasThom1}, $\Phi^n_{A,B}$ is a morphism in $\b{Ab}$, and moreover, the assignment
$(A,B)\mapsto\Phi^n_{A,B}$ defines a natural transformation $$\Phi^n:kk^{(n)}\to kk^{(n+1)}.$$
Thus, we have the following direct system of abelian groups:
$$\xymatrix{kk^{(1)}(A,B)\ar[r]^-{\Phi^1_{A,B}}&kk^{(2)}(A,B)\ar[r]^-{\Phi^2_{A,B}}&\cdots}.$$
The Corti\~{n}as-Thom bivariant K-group is defined to be the abelian group $$kk(A,B):=\lim_{n\to\infty}kk^{(n)}(A,B).$$
It is clear that $kk$ may be considered as a homotopy invariant functor
$$kk:\b{A}_\r{nc}^\r{op}\times\b{A}_\r{nc}\to\b{Ab}.$$
It is proved that there is a triangulated category $kk$ whose objects are those of $\b{A}_\r{nc}$ and whose morphism-sets are $kk(A,B)$.
Also, the assignment $f\mapsto\r{j}(f)$ defines a functor $\b{A}_\r{nc}\to kk$ which is universal among all
excisive, homotopy invariant and $\r{M}_\infty$-stable homology theories on $\b{A}_\r{nc}$ \cite[Theorem 6.6.2]{CortinasThom1}.

After the above long review of Corti\~{n}as-Thom's theory, we build a comparison morphism between $\f{Q}(C)$ and $kk(\bb{F},C)$:
For $C\in\b{A}_\r{nc}$, consider the composit
\begin{equation}\label{2012211536}
\xymatrix{\f{Q}(C)=[\r{q}\bb{F},\r{M}_\bullet(C)]\ar[r]&[\r{q}\bb{F},\r{M}_\infty(C)]\ar[r]^-{\r{j}}&kk(\r{q}\bb{F},\r{M}_\infty(C))
\cong kk(\r{q}\bb{F},C)}
\end{equation}
where the first arrow is induced by the canonical ind-morphism $\r{M}_\bullet(C)\to\r{M}_\infty(C)$.
As it is noted in \cite[$\S$6.1]{CortinasThom1} the sum operation in the $kk$-groups is the same as the operation induced by the abelian monoid
$\r{M}_\bullet$ in (\ref{2012211555}). Thus the composite (\ref{2012211536}) is a group-morphism.
For every surjective morphism $f$ in $\b{A}_\r{nc}$ it follows from \cite[Corollary 6.3.4]{CortinasThom1} that $\r{j}(f)$ has a right inverse
in $kk$. So, we have the following group-morphism induced by the inverse of $\r{j}(\r{ev}_1:\r{q}\bb{F}\to\bb{F})$:
\begin{equation}\label{2012212257}
kk(\r{q}\bb{F},C)\to kk(\bb{F},C)
\end{equation}
Composition of (\ref{2012211536}) and (\ref{2012212257}) is the canonically defined group-morphism
\begin{equation}\label{2012212303}
\f{Q}(C)\to kk(\bb{F},C).
\end{equation}
For Weibel's homotopy algebraic K-theory $KH$ we refer the reader to the original paper \cite{Weibel3} or \cite{Weibel1,Cortinas1}.
By \cite[Theorem 8.2.1]{CortinasThom1}, for any algebra $C$, we have a natural isomorphism between abelian groups $KH_0(C)$ and $kk(\bb{F},C)$.
Thus, (\ref{2012212303}) gives rise to the natural group-morphism
\begin{equation}\label{2012212315}
\f{Q}(C)\to KH_0(C).
\end{equation}
In case $C$ is a unital algebra, it can be seen from the proof of \cite[Theorem 8.2.1]{CortinasThom1} that the canonical morphism
(see \cite[$\S$5]{Cortinas1}),
\begin{equation}\label{2012212325}
K_0(C)\to KH_0(C)
\end{equation}
is equal to the composition of (\ref{2012202322}) and (\ref{2012212315}).
\begin{theorem}
\emph{Let $C$ be a unital $K_0$-regular algebra. Then (\ref{2012202322}) is injective and (\ref{2012212315}) is surjective.}
\end{theorem}
\begin{proof}
The theorem follows from the fact that by \cite[Proposition 5.2.3]{Cortinas1}, (\ref{2012212325}) is an isomorphism.
\end{proof}
We need the following elementary lemma.
\begin{lemma}\label{2011251455}
\emph{Let $C=(C_i)_i$ be in $\r{pro}\text{-}\b{C}$. Suppose that for every $i$, $C_i$ is a cocommutative cogroup object in $\b{C}$
and all structural morphisms of $C$ preserve the cogroup structures. Then $C$ is a cocommutative cogroup in $\r{pro}\text{-}\b{C}$
with comultiplication, counit, and coinverse induced by those of $C_i$'s in the obvious way.}
\end{lemma}
\begin{theorem}\label{2011272212}
\emph{For any algebra $A$, there exists an object $\ov{kk}_A\in\r{pro}\text{-}\r{Hot}(\b{A}_\r{nc})$ with a cocommutative cogroup structure
such that the abelian groups $[\ov{kk}_A,B]$ and $kk(A,B)$ are naturally isomorphic for every algebra $B$.}
\end{theorem}
\begin{proof}
Let $$\ov{kk}_A^{(n)}=\f{M}^\r{nc}(\r{J}^n(A),\c{M}_\infty(\bb{F}^{\c{S}^{n}})).$$
Then by Theorem \ref{2012041948} we have the natural bijection
\begin{equation}\label{2012221318}
kk^{(n)}(A,B)\cong[\ov{kk}_A^{(n)},B].
\end{equation}
Thus by Lemma \ref{2012041916}, $\ov{kk}_A^{(n)}$ has a cocommutative cogroup structure as an object in
$\r{pro}\text{-}\r{Hot}(\b{A}_\r{nc})$ that makes (\ref{2012221318}) into a natural isomorphism of groups.
By the Yoneda Lemma there exists a morphism $$\alpha_n:\ov{kk}_A^{(n+1)}\to\ov{kk}_A^{(n)}$$
in $\r{pro}\text{-}\r{Hot}(\b{A}_\r{nc})$ that induces the natural transformation $\Phi^n_{A,?}$.
Thus, $\alpha_n$ preserves the cogroup structures. Now, it follows from \ref{2011251455} that the inverse system
$$(\ov{kk}_A^{(n)},\alpha_n)_n$$ defines the desired object $\ov{kk}_A$.
\end{proof}
Following Phillips \cite{Phillips2}, $\ov{kk}_A$ may be called classifying homotopy pro-algebra.
The following theorem is one of the main results of this note.
\begin{theorem}\label{2012221244}
\emph{There exists an object $\ov{KH}_0$ in $\r{pro}\text{-}\r{Hot}(\b{A}_\r{nc})$ with a cocommutative cogroup structure
such that the abelian groups $[\ov{KH}_0,B]$ and $KH_0(B)$ are naturally isomorphic for every algebra $B$.}
\end{theorem}
\begin{proof}
It follows directly from \cite[Theorem 8.2.1]{CortinasThom1} and Theorem \ref{2011272212}.
\end{proof}
%%%%%%%%%%%%%%%%%%%%%%%%%%%%%%%%%%%%%%%%%%%%%%%%%%%%%%%%%%%%%%%%%%%%%%%%%%%%%%%%%%%%%%%%%%%%%%%%%%%%%%%%%%%%%%%%%%%%%%%%%%
%%%%%%%%%%%%%%%%%%%%%%%%%%%%%%%%%%%%%%%%%%%%%%%%%%%%%%%%%%%%%%%%%%%%%%%%%%%%%%%%%%%%%%%%%%%%%%%%%%%%%%%%%%%%%%%%%%%%%%%%%%
%%%%%%%%%%%%%%%%%%%%%%%%%%%%%%%%%%%%%%%%%%%%%%%%%%%%%%%%%%%%%%%%%%%%%%%%%%%%%%%%%%%%%%%%%%%%%%%%%%%%%%%%%%%%%%%%%%%%%%%%%%
%%%%%%%%%%%%%%%%%%%%%%%%%%%%%%%%%%%%%%%%%%%%%%%%%%%%%%%%%%%%%%%%%%%%%%%%%%%%%%%%%%%%%%%%%%%%%%%%%%%%%%%%%%%%%%%%%%%%%%%%%%
%%%%%%%%%%%%%%%%%%%%%%%%%%%%%%%%%%%%%%%%%%%%%%%%%%%%%%%%%%%%%%%%%%%%%%%%%%%%%%%%%%%%%%%%%%%%%%%%%%%%%%%%%%%%%%%%%%%%%%%%%%
%%%%%%%%%%%%%%%%%%%%%%%%%%%%%%%%%%%%%%%%%%%%%%%%%%%%%%%%%%%%%%%%%%%%%%%%%%%%%%%%%%%%%%%%%%%%%%%%%%%%%%%%%%%%%%%%%%%%%%%%%%
%%%%%%%%%%%%%%%%%%%%%%%%%%%%%%%%%%%%%%%%%%%%%%%%%%%%%%%%%%%%%%%%%%%%%%%%%%%%%%%%%%%%%%%%%%%%%%%%%%%%%%%%%%%%%%%%%%%%%%%%%%
%%%%%%%%%%%%%%%%%%%%%%%%%%%%%%%%%%%%%%%%%%%%%%%%%%%%%%%%%%%%%%%%%%%%%%%%%%%%%%%%%%%%%%%%%%%%%%%%%%%%%%%%%%%%%%%%%%%%%%%%%%
%%%%%%%%%%%%%%%%%%%%%%%%%%%%%%%%%%%%%%%%%%%%%%%%%%%%%%%%%%%%%%%%%%%%%%%%%%%%%%%%%%%%%%%%%%%%%%%%%%%%%%%%%%%%%%%%%%%%%%%%%%
%%%%%%%%%%%%%%%%%%%%%%%%%%%%%%%%%%%%%%%%%%%%%%%%%%%%%%%%%%%%%%%%%%%%%%%%%%%%%%%%%%%%%%%%%%%%%%%%%%%%%%%%%%%%%%%%%%%%%%%%%%
\end{document}